\documentclass[11pt]{amsart}
\usepackage{amsfonts}
\usepackage{amssymb}
\usepackage{graphics}
\usepackage{amsmath, amsthm, latexsym}

\def \Vol{\textup{Vol}}
\def \L{\mathcal{L}}

\theoremstyle{plain}
\newtheorem{Th}{Theorem}[section]
\newtheorem{Lem}[Th]{Lemma}
\newtheorem{Prop}[Th]{Proposition}
\newtheorem{Cor}[Th]{Corollary}

\theoremstyle{definition}
\newtheorem{Ex}[Th]{Example}
\newtheorem{Def}[Th]{Definition}
\newtheorem{Rem}[Th]{Remark}

\begin{document}
\title{Convex bodies and algebraic equations on affine varieties}
\author{Askold Khovanskii\\Kiumars Kaveh\\Department of Mathematics\\
University of Toronto}
\maketitle
\tableofcontents


\noindent{\bf \large Note:} This is a preliminary version and may contain several typos.

\section{Introduction} \label{sec-intro}
The theory of Newton polytopes relates algebraic geometry of
subvarieties in $(\Bbb C^*)^n$  and convex geometry (for a survey see
for example \cite{Khovanskii-formulas}). In other words, this is a connection between
the theory of toric varieties and geometry of convex polytopes. In this
paper we discuss a much more general connection between algebraic geometry and
convex geometry. This connection is useful in both directions. It
yields new, simple and transparent proofs of a series of classical
results (which are not considered as simple) both in algebraic geometry
and in convex geometry.

We prove the following classical results from algebraic geometry:
Hodge Index Theorem (according to which the square of intersection index of two
algebraic curves on an irreducible algebraic surface is greater than or
equal to the product of self intersections of the curves),
Kushnirenko--Bernstein theorem on the number of roots of generic
system of algebraic equations with fixed Newton polyhedra. We also
develop a version of intersection theory for (quasi) affine varieties. We
show that properties of number of solutions of a generic system of
equations on an $n$-dimensional (quasi) affine algebraic variety resemble, in
many ways, the properties of mixed volumes of $n$ convex bodies in $\Bbb{R}^n$.
In the part related to convex geometry we prove Alexander--Fenchel
inequality --- which is one of the main inequalities concerning mixed volumes. Many
other geometric inequalities follow as its corollaries.

In our proofs we use simple and rather restricted tools. From
algebraic geometry we use classical Hilbert theory on degree of
subvarieties of projective space (see Section \ref{subsec-13} for statement of Hilbert theorem,
its proof could be found in most of the textbooks in algebraic geometry for example
\cite[Lecture 13]{Harris}).
From convex geometry we use Brunn-Minkowski inequality. It is actually enough for us to use
the classical isoperimetric inequality which is Brunn-Minkowski for convex domains in plane.

About the content of the paper. In  Sections 1-7 we construct a version of
intersection theory for (quasi) affine varieties. To a (quasi) affine variety $X$ we associate
a set $K(X)$. By definition each element in $K(X)$ is a finite
dimensional space $L$ of regular functions on $X$, such that for any
point in $X$ at least one function from $L$ is not equal to zero.
Product $L_1L_2$ of two spaces  $L_1, L_2\in K(X)$ is the space spanned
by the functions $f_1f_2$, where $f_1\in L_1$, $f_2\in L_2$. The set
$K(X)$ equipped with this multiplication become a commutative
semigroup. In Section \ref{subsec-3} we introduce an intersection index in the
semigroup $K(X)$, where now $X$ is an irreducible (quasi) affine $n$-dimensional variety.  The intersection index of an $n$-tuple $L_1,\dots,L_n\in K(X)$, denoted by $[L_1,\dots,L_n]$, is the number of
solutions of a sufficiently general system of equations
$f_1=\dots=f_n=0$ on $X$, where $f_1\in L_1,\dots,f_n\in L_n$.

We show that for almost all $n$-tuples $f_1\in L_1,\dots,f_n\in
L_n$, number of solutions of the system $f_1=\dots=f_n=0$ is the
same and hence the intersection index is well-defined.
In Section \ref{subsec-2} we state some classical results which we need for the proof of this fact. The
properties of the intersection index are similar to the properties
of mixed volumes of $n$ convex bodies. Some of these properties
could be deduced from the case in which $X$ is an algebraic curve.
(see Sections~4-6). But to prove  the most interesting property, namely
an analogue of Alexandrov--Fenchel inequality, we have to consider
algebraic surfaces (see Section \ref{subsec-7}). The corresponding  property of
algebraic surfaces is proved Section \ref{subsec-17}.

In Section \ref{subsec-14} we associate a convex body to each space $L \in K(X)$
where $X$ is an irreducible $n$-dimensional (quasi) affine variety.
We will show that, under some small extra assumptions,
the volume of this convex body multiplied by $n!$ is equal to the self intersection index
$[L,\dots,L]$ of the space $L$. This construction provides the
relation between algebraic geometry and convex geometry in this paper.
Let us describe this construction more precisely.

First we fix a $\Bbb Z^n$-valued valuation on the field of rational
functions on $X$. There are many different valuations of
this kind (see Section \ref{subsec-12}). Different valuations associate different
convex bodies to a space $L \in K(X)$. The body is constructed as follows:
for each $k \in {\Bbb N}$, values of the valuation on the space $L^k$
belong to a finite subset $\tilde G_k(L)$ in the group $\Bbb Z^n$. The
number of points in the set $\tilde G_k(L)$ is equal to the
dimension of the space $L^k$. Let us add a first coordinate equal
to $k$ to all points in $\tilde G_k(L)$. We obtain a new set
$G_k(L) \subset \Bbb Z\times\Bbb Z^n$. The
union over $k$ of all the sets $G_k(L)$ is a semigroup $G(L)$ in
$\Bbb Z\times\Bbb Z^n$.  Let us consider the smallest convex cone $C$ (centered at origin)
in ${\Bbb R}\times {\Bbb R}^n\supset\Bbb Z\times \Bbb Z^n$, which
contains the semigroup $G(L)$. The intersection of the cone $C$ with
the hyperplane  $k=1$ is our desired convex body  $\Delta (G(L))$
associated with the space  $L$. We will show that the set of all integral
point in the cone $C$ provides a very good approximation of the
semigroup  $G(L)$. After that the relation between the volume of the
body $\Delta (G(L))$ and the self intersection $[L,\dots,L]$ follows
from the Hilbert theorem (see Section \ref{subsec-13}).

The usual Newton polytope associated to a Laurent polynomial is a
very special case of this construction (see Section \ref{subsec-15}).
The Newton polytopes are naturally related to toric varieties.
Interestingly, the Gelfand-Cetlin polytopes of irreducible representations
of $\textup{GL}(n, \Bbb C)$, and more generally string polytopes for irreducible
representations of a connected reductive group $G$, also appear as the Newton convex
body $\Delta(G(L))$. As $X$ we take the flag variety, or rather the open affine
Schubert cell in it (Example \ref{ex-GC}). For this, see \cite{Okounkov-Newton-polytope}
for Gelfand-Cetlin polytopes of $G=\textup{SP}(2n, \Bbb C)$ and \cite{Kaveh} for the general case.

The results we need on semigroups of integral points are proved in
Sections \ref{subsec-10}-\ref{subsec-11}. In Section \ref{subsec-10} we not only estimate the number of
integral points with fixed first coordinate but also estimate the
sum of value of a polynomial over this subset in the semigroup $G(L)$. We won't
need this estimation of the sum of values of a polynomial in this paper,
although it will be used in the next paper \cite{Khovanskii-Kaveh}. In
\cite{Khovanskii-Kaveh} we consider a variety  $X$, equipped  with a reductive group
action and a subspace $L$ of regular functions on $X$ 
invariant under this action. We will prove a generalization of Kazarnovskii-Brion
formula (for the degree of a normal projective spherical variety) to (quasi) affine,
not necessarily normal, spherical varieties.

The results of Section \ref{subsec-10}-\ref{subsec-11}
use the facts from convex geometry which we prove in Section \ref{subsec-9}.

In Section \ref{subsec-15}
we briefly show that the well-known Kushnirenko and Bernstein theorems follow
from our general results. In fact the proof in Section \ref{subsec-15}
almost coincides  with the proof in  \cite{Khovanskii-finite-sets}. Bernstein theorem
relates mixed volume with number of solutions of a generic system of
Laurent polynomial equations.  In Section \ref{subsec-17} using the isoperimetric inequality
(Brunn--Minkowski inequality for planar convex bodies) we prove an
algebraic analogous of Alexandrov--Fenchel inequality and its
numerous corollaries. In Section \ref{subsec-18} we show that the
the corresponding geometric inequalities follows from their algebraic analogues.

If $X$ is an affine algebraic curve one can describe the
geometry of the semigroup $G(L)$,  $L\in K(X)$ in detail (see
Section \ref{subsec-19}).

We should point out that the assumption that $X$ is (quasi) affine is not crucial for
the results of this paper. In fact, one can take $X$ to be any irreducible variety and
replace $L$ with a subspace of section of a line bundle on $X$. Given a
valuation on the ring of sections of the line bundle, in the same way
one constructs a convex body associated to $(X, L)$. The same arguments as in the paper
then can be used to give a relation between the volume of this convex body and the
self-intersection number of a generic section from $L$.

This paper is our first work in a series of papers under preparation,
dedicated to the new relation between convex geometry and algebraic geometry.

\section{An intersection theory for affine varieties}
\subsection{Semigroup of subspaces of a ring of functions on a set}  \label{subsec-1}
We start with some general definitions.  {\it A set equipped with a
ring of functions} is a set $X$, with a ring $R(X)$ consisting of complex
valued functions, containing all complex constants. To a pair
$(X, R(X))$ one can associate the set  $VR(X)$ whose elements are
vector subspaces in $R(X)$.

There is a natural multiplication in $VR(X)$. For any two
subspaces $L_1, L_2\subset R(X)$ define {\it product
$L_1L_2$} to be the linear span of functions $fg$, where $f\in L_1$ and
$g\in L_2$. With this product the set $VR(X)$ becomes a
commutative semigroup.

Let us say that a subspace $L$ {\it has no common zeros on $X$}, if
for each $x\in X$ there is a function $f\in L$ with $f(x)\neq
0$.

\begin{Prop} \label{1.1}
Let $L_1,L_2 $ be vector subspaces
in $R(X)$. If  $L_1,L_2 $ have finite dimension (respectively, if each subspace
$L_1,L_2 $ has no common zeros on $X$), then the space $L_1L_2$ is
finite dimensional (respectively, the space $L_1L_2$ has no common zeros on $X$).
\end{Prop}

\begin{proof} 1) Let $\{ f_i\}$, $\{ g_j\}$ be  bases for
subspaces $L_1$,$L_2$. Then the functions $\{f_ig_j \}$ span the space
$L_1L_2$. So, if $L_1,L_2$ have finite dimension then,  $L_1L_2$
also has a finite dimension.  2) If the functions
$f_1\in L_1$, $f_2\in L_2$ do not vanish at a point $x \in X$, then the function
$f_1f_2\in L_1L_2 $ does not vanish at $x$ and thus if each space
$L_1,L_2$ has no common zeros on $X$, then the space $L_1L_2$ also
has no common zeros on $X$.
\end{proof}

According to Proposition \ref{1.1},  {\it subspaces of finite
dimension in $R(X)$, each of which has no common zeros on $X$ form a
semigroup in $VR(X)$ which we will denote by $KR(X)$}.

Assume that $Y\subset X$ and that the restriction of each function
$f\in R(X)$ to the set $Y$ belongs to a ring $R(Y)$.
We will denote the restriction of a subspace $L \subset R(X)$ to $Y$
by the same symbol $L$. Clearly if  $L\in KR(X)$ then
$L\in KR(Y)$.

In this paper we will not use general sets equipped by rings of functions.
Instead the following example plays a main role.

\begin{Ex} Let $X$ be a complex (quasi) affine algebraic variety
and let $R(X)$ be the ring of regular functions on $X$.
In this case to make notations shorter we will not mention the ring
$R(X)$ explicitly and the semigroup $KR(X)$ will be denoted by $K(X)$.
\end{Ex}

Any subspace $L \in K(X)$ gives a natural map $\Phi_L: X \to \Bbb P(L^*)$, where
$L^*$ denotes the vector space dual of $L$. For $x \in X$ define $\xi \in L^*$
by $$\xi(f) = f(x),$$ for all $f \in L$. Since the elements of $L$ have no common zero,
$\xi \neq 0$. Let $\Phi_L(x)$ be the
point in $\Bbb P(L^*)$ represented by $\xi$. Fix a basis $\{f_1, \ldots f_d\}$ for $L$.
One verifies that the map $\Phi_L$ in the homogeneous coordinates in $\Bbb P(L^*)$,
corresponding to the dual basis to the $f_i$, is given by $$ \Phi_L(x) = (f_1(x) : \cdots: f_d(x)).$$
A subspace $L \in K(X)$ is called {\it very ample} if $\Phi_L$ is an embedding.

Finally let us
say that a regular function  $f \in R(X)$ {\it
satisfies an integral algebraic equation over a space $ L\in K(X)$},
if $$f^m+a_1f^{m-1}+\dots +a_m=0,$$
where  $m$ is a natural number and $a_i \in L^i$, for each $i=1,\ldots, m$.

\subsection{Preliminaries on affine algebraic varieties} \label{subsec-2}
Now we discuss some facts needed to define an intersection
index in the semigroup $K(X)$. We will need particular cases of the
following results: 1) An affine algebraic variety has a {\it finite}
topology; 2) There are finitely many topologically different
varieties in an  algebraic family of  affine varieties; 3) In such a family
the set of parameters, for which the corresponding members have the same topology,
is a complex semi-algebraic subset in the  space of parameters;
4) A complex semi-algebraic subset in a vector space covers almost
all of the space, or  covers only a very small part of it.

Now let us give exact statements of these results and their
particular cases  we will use.

Let $X, Y$ be  complex affine algebraic varieties and let
$\pi: X\rightarrow Y$ be a regular map. Consider a family of affine
algebraic varieties  $X_y=\pi^{-1}(y)$, parameterized by points $y\in
Y$. The following theorem is well-known.

\begin{Th} Each variety $X_{y}$ has a homotopy type of a
finite $CW$-complex. There is a finite  stratification of the
variety $Y$ into complex semi-algebraic strata  $Y_{\alpha}$, such
, that for points $y_1, y_2$ belonging to the same stratum $Y_{\alpha}$
varieties $X_{y_1}, X_{y_2},$ are homeomorphic (In particular, in
the family  $X_y$ there are one finitely many topologically
different varieties.)
\end{Th}

When $X, Y$ are real affine algebraic varieties and
$\pi:X\rightarrow Y$ is a regular real map, a similar statement
holds. One can also extend it to some other cases of varieties and
maps (see \cite{Dries}). We will need only the following simple corollary of this
theorem for which we give sketch of a proof (independent of the above theorem).

Let $L_1,\dots, L_n $ be finite dimensional subspaces in the space
of  regular functions on an $n$-dimensional complex affine
algebraic variety $X$. Denote by $X_{\bold f}$, where $\bold
f={(f_1,\dots,f_n)}$ is a point in $\bold L=L_1\times\dots\times
L_n$, the subvariety of $X$, defined by the
system of equations $f_1=\dots=f_n=0$. In the space $\bold L$
of parameters consider the subset $\bold F$ consisting of all
parameters $\bold f$ such that the set  $X_{\bold f}$ contains
isolated points only.

\begin{Cor} \label{2.1}
{\it 1) If $\bold f\in \bold F$, then the
set $X_{\bold f}$ contains finitely many points. Denote the number of
points in $X_{\bold f}$ by $ k(\bold
f)$; 2) Function $ k(\bold f)$ on the set $\bold F$ is
bounded; 3) The subset $\bold F_{\max}\subset \bold F$ on which the
function $k(\bold f)$ attains its maxima is a complex semi-algebraic
subset in $\bold L$.}
\end{Cor}
\begin{proof}[Sketch of proof (independent of above theorem)] One can assume, that the
variety $X$ is defined  in  a space $\Bbb C^N$ by a non degenerated
system of  polynomial equations $g_1=\dots =g_{N-n}=0$. (To be
exact, $X$ can be covered by a finite collection of  Zariski open
domains and in each domain $X$ is defined in such a way, see \cite{Whitney}.
It is then enough to estimate the number of  roots in each
domain). One can also assume that all the functions belonging to the
spaces $L_1, \dots,L_n$  are restrictions of polynomials on
$\Bbb C^N$ to $X$ belonging to a finite dimensional space $\bar L$.
Let $M$ be the maximum degree of all the polynomial $g_i$ and the all
polynomials in $\bar L$. From the classical Bezout theorem it is
easy to deduce that $k(\bold f)< M^N$. Using the complex version of
Tarski theorem one can prove that the function $k(\bold f)$ takes finitely many values $c$ and each
level set $\bold F_{c}$ is semi-algebraic. (Similar fact is true in real algebraic
geometry. One proves it using  Tarski theorem. For an elementary
proof of Tarski theorem see \cite{Burda-Khovanskii}).
\end{proof}

We will need the following simple property of complex semi-algebraic
sets.
\begin{Prop} Let $F\subset L$ be a complex
semi-algebraic subset in a  vector space $L$. Then either there is an
algebraic hypersurface  $\Sigma\subset L$ which contains $F$,
or $F$ contains a Zariski open set  $U\subset L$.
\end{Prop}

We will use this proposition in the following form.

\begin{Cor} \label{2.2}
{\it Let $F\subset L$ be a complex
semi-algebraic subset in a vector space  $L$. In the following cases
$F$ contains a (non-empty) Zariski open subset  $U\subset L$: 1) $F$ is an
everywhere  dense  subset of $L$, 2) $F$ does not have  zero
measure.}
\end{Cor}

\subsection{An intersection index in semi-group $K(X)$} \label{subsec-3}
\begin{Def} Let $X$ be a complex $n$-dimensional
(quasi) affine algebraic variety and let $L_1,\dots, L_n$ be elements of
$K(X)$. The {\it intersection index} $[L_1,\dots, L_n]$ of
$L_1,\dots, L_n\in K(X)$ is the maximum of number of roots of a
system $f_1=\dots=f_n$ over all the points $\bold f=(f_1,\dots,f_n)\in
L_1\times\dots\times L_n=\bold L$, for which corresponding system
has finitely many solutions.
\end{Def}

By Corollary \ref{2.1} the maximum is attained and  the previous definition
is well-defined.

\begin{Th}[Obvious properties of the intersection
index] \label{3.1}
(1) $[L_1,\dots,L_n]$ is a symmetric function of
the n-tuples $L_1,\dots,L_n$ { (i.e. takes the same value under a
permutation of the elements $L_1,\dots,L_n$)}, {(2) it is
monotone}, { (i.e. if $L'_1\subset L_1,\dots, L'_n\subset L_n$,
then $[L_1,\dots,L_n]\geq [L'_1,\dots,L'_n])$} and {(3) non-negative}
(i.e. $[L_1,\dots,L_n]\geq 0$).
\end{Th}

Theorem \ref{3.1} is a straight forward corollary from the definition.

Let $X$ be a complex $n$-dimensional (quasi) affine algebraic variety, let
$k \in {\Bbb N}$ and let  $L_1,\dots, L_k$ be an $n$-tuple
of subspaces belonging to the semigroup $K(X)$. Put $\bold
L=L_1\times \dots \times L_k$.

\begin{Prop} \label{3.2}
There is a Zariski open domain $\bold U$
in $\bold L$ such that for each point $\bold f =(f_1,\dots,f_k)$ in
$\bold U$ the system of equations  $f_1=\dots =f_k=0$ on $X$ is non
degenerate (that is, at each root of the system the covectors
$df_1,\dots,df_k$ are linearly independent).
\end{Prop}

\begin{proof} Fix  a basis $\{g_{i,j}\}$ for each space $L_i$.
Consider all the $k$-tuples $ \bold g _{\bold j}=( g_{1,j_1},\dots,
g_{k,j_k})$, where $ \bold j=(j_1,\dots,j_k)$, containing exactly
one vector from each of the bases for the $L_i$. Denote by $V_{\bold j}$
the Zariski open domain in $X$ defined by the system  of inequalities $
g_{1,j_1}\neq 0 ,\dots, g_{k,j_k}\neq 0$. The union of the sets
$V_{\bold j}$ coincides with $X$, because $L_1,\dots,L_k \in
K(X)$. In the domain $V_{\bold j}$ rewrite the system
$f_1=\dots =f_k=0$ as follows: represent each function $f_i$ in the
form $\bar f_i=\bar f_i +c_ig_{i,j_i}$, where $\bar f_i$ belongs to
the linear span of all the vectors $g_{i,j}$ excluding the vector
$g_{i,j_i}$. Now in  $V_{\bold j}$ the system could be
rewritten as  $\frac {\bar f_1}{g_{1,j_1}}=-c_1, \dots,
\frac {\bar f_k}{g_{k,j_k}}=-c_k$. According to Sard's theorem,
for almost all the $\bold c =(c_1,\dots, c_k)$ the system is non-degenerate.
Denote by $W_{\bold j}$ the subset in $\bold L$,
consisting of all $\bold f$ such that the system
$f_1=\dots =f_k=0$ is non-degenerate in $V_{\bold j}$. We have proved
that the set $W_j$ is a set of full measure in $\bold L$. On the
other hand the set $W_j$ is a complex semi-algebraic subset in
$\bold L$. Thus, according to Corollary \ref{2.2},  $W_j$ contains a
Zariski open subset $\bold U_{\bold j}$. The intersection $\bold U$
of the sets  $\bold U_{\bold j}$ is a Zariski open subset in $\bold
L$ which satisfies all the requirements of Proposition \ref{3.2}.
\end{proof}

\begin{Prop} \label{3.3}
The number of isolated roots of a
system $f_1=\dots=f_n=0$, where $f_1\in L_1,\dots, f_n\in L_n$,
counted with multiplicity, is smaller than or equal to
$[L_1,\dots,L_n]$.
\end{Prop}
\begin{proof}
Let $A$ be the set of isolated roots of our system.
Let $k(A)$ be the sum of multiplicities of roots in $A$.
According to Proposition \ref{3.2} one can
perturb the system a little bit to make it non-degenerate.
Under such a perturbation the roots belonging to the set $A$  will
split into $k(A)>[L_1,\dots,L_n]$ simple roots. Thus we get a non-degenerate
system with finitely many simple roots. By Corollary
\ref{2.1} the number of these roots can not be bigger than $[L_1,\dots,L_n]$.
\end{proof}

Now we prove that if a system of equations is generic then instead of
inequality in Proposition \ref{3.3} we have an equality.
As before let $\bold L = L_1\times \dots \times L_n$.

\begin{Prop} \label{3.4}
There is a Zariski open domain $\bold U$
in $\bold L$ such that for each point $\bold f =(f_1,\dots,f_n)$ in
$\bold U$ the system of equations  $f_1=\dots =f_n=0$ on $X$ is non
degenerate and has exactly $[ L_1,\dots,L_n]$ solutions.
\end{Prop}
\begin{proof}{Proof}  First, if a system has $[L_1,\dots,L_n]$ many isolated roots
than the system is non-degenerate, otherwise its number of roots counting with
multiplicity is bigger than  $[L_1,\dots, L_n]$, which is
impossible by Proposition \ref{3.3}. So there must be a non-degenerate system which has
$[L_1,\dots,L_n]$ solutions. Second, any sufficiently general system
has exactly the same number of isolated roots and almost all
sufficiently general systems are non degenerate. So the set of non
degenerate  systems which have exactly $[L_1,\dots,L_n]$ could not
be a set of measure zero. But this set is complex semi-algebraic,
so according to the corollary 2.2 it contains a Zariski open domain.
\end{proof}

For each $k$-dimensional (quasi) affine subvariety $Y$ in $X$ and for each
$k$-tuple of spaces  $L_1,\dots,L_k \in K(X)$ let
$[L_1,\dots,L_k]_Y$  be the intersection index of the restrictions of
these subspaces to $Y$.

Consider an $n$-tuple  $L_1,\dots,L_n\in K(X)$.  For $k \leq n$
put $\bold L(k) = L_1\times\dots\times L_k$.
According to Proposition \ref{3.2} there is a Zariski open subset
$\bold U(k)$ in $\bold L(k)$ such that if $\bold f(k)
=(f_1,\dots,f_k)\in \bold U(k)$ then the system $f_1=\dots=f_k=0$
is non-degenerate and hence defines a smooth subvariety   $X_{\bold f(k)}$
in $X$.

\begin{Th} \label{3.5}
1) For each point  $\bold f(k)\in
\bold U(k)$ the following inequality holds
\begin{equation}
[L_1,\dots L_n]_X\leq [L_{k+1},\dots L_n]_{X_{\bold f(k)}}.
\end{equation}
2) There is a Zariski open subset $\bold V(k) \subset \bold U(k)$,
such that for each point $\bold f(k)\in V(k)$  the inequality $(1)$ in fact is an equality.
\end{Th}
\begin{proof} 1) If for a point $\bold f(k)$ inequality $(1)$
does not hold, then there are  $f_{k+1}\in L_{k+1}, \dots, f_{n}\in
L_{n}$  such that the system $f_1=\dots=f_k=f_{k+1}=\dots=f_n=0$ has
more isolated solution on $X$ than the intersection index $[L_1,
\dots,L_n ]$, which is impossible. 2) According to Proposition \ref{3.2}
the collection of systems $\bold f
=(f_1,\dots,f_ n)\in \bold L$ for which the subsystem $f_1=\dots=f_k=0$ is
non-degenerate contains a Zariski
open domain  $\bold V\subset \bold L$. Let $\pi:\bold L \rightarrow
\bold L(k)$  be the projection $(f_1,\dots,f_n) \mapsto (f_1,\dots,f_k)$. Now
we can take $\bold V(k)$ to be any  Zariski open domain in $\bold L(k)$
contained in $\pi (\bold V)$.
\end{proof}

Theorem \ref{3.5} allows us to reduce the computation of the intersection index on
a high dimensional (quasi) affine variety  to computation of the
intersection index on a lower dimensional (quasi) affine subvariety. It is not
hard to establish main properties of the intersection index for affine
curves. Using Theorem \ref{3.5} we will then obtain corresponding properties for
the intersection index on (quasi) affine varieties of arbitrary dimension for free.

\subsection{Preliminaries on affine algebraic curves} \label{subsec-4}
Here we present some basic facts about affine algebraic curves which we will use later.
Let $X$ be a smooth complex
affine  algebraic curve (not necessarily irreducible).

\begin{Th}[normalization of algebraic curves]
There is a unique (up to isomorphism) smooth projective curve $\bar X$
which contains $X$.
The complement  $A = \bar X \setminus X$, is a finite set, and any
regular function on  $X$ has a  meromorphic  extension  to $\bar
X$.
\end{Th}

One can find a proof of this  classical  result in most of the text books in
algebraic geometry (e.g. \cite[Chapter 1]{Hartshorne}).
This theorem allows us to find the number of
zeros of a regular function  $g$ on $X$ which has a
prescribed behavior at infinity i.e. $\bar X \setminus X$.
Indeed if  $g$ is not identically
zero on some irreducible component of the curve $X$, then the order $ord_a g$ of
its meromorphic extension at a point $a\in \bar X$ is well-defined defined.
Function $g$ on the projective curve $\bar X$ has the same number of
roots (counting with multiplicities) as the number of poles
(counting with multiplicities). Thus we have the following.

\begin{Prop} \label{4.1}
For every  regular  function $g$ on
an affine algebraic curve $X$ (which is not identically
zero at any irreducible component of $X$) the number of roots
counting with multiplicity  is equal to  $- \sum_{a\in A}
ord_a g$, where $ord_a g$ is the order at the point $a$ of the
meromorphic extension of the function  $g$ to $\bar X$.
\end{Prop}

\subsection{Intersection index in semigroup $K(X)$ of an affine
algebraic variety $X$} \label{subsec-5}
Let $L\in K(X)$ and let $B=\{ f_i\}$ be a basis for
$L$ such that none of the $ f_i$  are identically equal to
zero at any component of the curve $X$. For each point $a\in A=\bar
X\setminus X$ denote by $ord_a L$ the minimum, over all functions in
$B$, of the numbers  $ord_a f_i$. Clearly for every $g \in L$ we have $ord_a g \geq ord_a L$.
The collections of functions $g\in L$ whose order at
the point  $a$ is strictly bigger than  $ord_a L$ form a proper
subspace $L_a$ of $L$.

\begin{Def} By definition {\it degree} of a subspace $L\in K(X)$
is $\sum\limits _{a\in A} -ord_a L$, and denoted by $\deg(L)$.
\end{Def}

For each component  $X_j$ of the curve $X$ denote the
subspace in $L$, consisting of all the functions identically zero
on  $X_i$ by $L_{X_i}$.  The space $L_{X_i}$ is a proper subspace in  $L$ because
$L\in K(X)$.

The following is a corollary of Proposition \ref{4.1}.
\begin {Prop} \label{5.1}
If function  $f\in L$ does not
belong to the union of the subspaces $L_{X_i}$, then $f$ has
finitely many roots on $X$. The number of the roots of the function
$f$, counted with multiplicity, is less than or equal to $\deg(L)$.
If function $f$ is not in the union of
the subspaces $L_a$, $a\in A$, then the equality holds.
\end{Prop}

\begin{Prop} \label{5.2}
For any two $L, G \in K(X)$
the following identity holds
$[L]+[G]=[LG].$
\end{Prop}
\begin{proof} For each point $a\in A$ and any two functions $f \in L$,
$g \in G$ the identity $ord_a f+ord _a g = ord_a fg$ holds. As a
corollary we have $ord_a L + ord_a G= ord_a LG$. So,
$\deg(L) +\deg(G) = \deg(LG)$ and hence $[L]+[G]=[LG]$.
\end{proof}

Consider the map $-Ord$  which associate to a subspace $L\in K(X)$
an integral valued function on the set $A$, namely value of $-Ord(L)$
at $a\in A$ is equals $-ord_a L$. The map $-Ord$ is a homomorphism from
the multiplicative semigroup $K(X)$ to the additive group of
integral valued functions on the set $A$. Clearly the number $[L]$ can be
computed in terms of the homomorphism $-Ord$ because $[L]= \deg(L)=
\sum_{a\in A} -ord _aL $.

\begin{Prop} \label{5.3}
Assume that a regular function $g$
on the curve $X$ satisfies an integral algebraic equation over a
subspace $L\in K(X)$. Then at each point $a\in A$ we have
$$ord _a g \geq ord _a L.$$
\end{Prop}
\begin{proof} Let $ g^n +f_1g^{n-1}+\dots +f_n=0$ where
$f_i \in L^i$. Suppose $ord _a g =k<ord _a L$.
Since $g^n = -f_1g^{n-1} - \cdots - f_n$ we have
$nk = ord_a g^n \geq \min\{ord_a f_1g^{n-1}, \ldots, f_n\}$.
That is, for some $i$, $nk \geq ord_a f_i + k(n-i)$.
But for every $i$, $ord_a f_i g^{n-i} = ord_a f_i + ord_a g^{n-i}
> i \cdot ord_a L + k(n-i) > nk$. The contradiction proves the claim.
\end{proof}

\begin{Cor} \label{5.4}
Assume that a regular function $g$ on
the curve $X$ satisfies an integral  algebraic equation over a
subspace  $L\in K(X)$. Consider the subspace $G \in K(X)$ spanned by
$g$ and $L$. Then: 1) At each point
$a\in A$ the equality $ord _a L=ord_a G$ holds; 2) $[L]=[G]$; 3) For
each subspace $M \in K(X)$ we have $[LM]=[GM]$.
\end{Cor}

\subsection{Properties of the intersection index which can be deduced from the curve case}
\label{subsec-6}
\begin{Th}[Multi-linearity] \label{6.1}
Let $L_1', L_1'', L_2, \ldots, L_n \in K(X)$
and put $L_1= L_1'L_1''$. Then
$$[L_1,\dots,L_n]=[L''_1,\dots,L_n]+[L'_1,\dots,L_n].$$
\end{Th}
\begin{proof} Consider three $n$-tuples $(L'_1,\dots,L_n)$,
$(L''_1,\dots,L_n)$ and $(L'_1L''_1,\dots,L_n)$ of elements of the
semigroup $K(X)$. According to the Theorem \ref{3.5} there is an $(n-1)$-tuple
$f_2\in L_2,\dots, f_n\in L_n$, such that the system
$f_2=\dots+f_n=0$ is non-degenerate and defines a curve
$Y\subset X$ such that
$[L'_1,\dots,L_n]=[L_1']_Y$, $[L''_1,\dots,L_n]=[L_1'']_Y$ and
$[L'_1L''_1,\dots,L_n]=[L_1'L''_1]_Y$. Using Proposition \ref{5.2} we now
obtain $[L_1'L''_1]_Y= [L_1']_Y+[L''_1]_Y$ and theorem is
proved.
\end{proof}

\begin{Th}[Integral closure property]
\label{6.2}
Let $L_1 \in K(X)$ and let
$G_1\in K(X)$ be a subspace spanned by $L_1 \in K(X)$ and
some regular functions $g$ satisfying an
integral algebraic equation over $L_1$. Then for any
$(n-1)$-tuple $L_2,\dots,L_n\in K(X)$ we have
$$[L_1,L_2,\dots,L_n]=[G_1,L_2,\dots,L_n].$$
\end{Th}
\begin{proof} Consider two $n$-tuples $(L_1,L_2,\dots,L_n)$
$(G_1,L_2,\dots,L_n)$ of $K(X)$. According
to Theorem \ref{3.5} there is a $(n-1)$-tuple $(f_2, \ldots, f_n)$, $f_i \in
L_i$, such that the system $f_2=\dots+f_n=0$ is non degenerate and
defines a curve $Y\subset X$ such that
$[L_1,L_2,\dots,L_n]=[L_1]_Y$, $[G_1,L_2 \dots,L_n]=[G_1]_Y$. Using
Corollary \ref{5.4} we obtain $[L_1]_Y= [G_1]_Y$ as required.
\end{proof}

\subsection{Properties of the intersection index which can deduced
from the surface case} \label{subsec-7}

Let $Y\subset X$ be a (quasi) affine subvariety and let  $L\in K(X)$ be a
very ample subspace. Then the restriction of functions from $L$ to $Y$
is a very ample space in  $K(Y)$.

\begin{Th}[A version of Lefschetz theorem]
Let $X$ be a smooth irreducible $n$-dimensional (quasi) affine variety
and let  $L_1,\dots,L_k\in K(X)$, $k<n$, be very ample subspaces, i.e.
the maps $\Phi_{L_i}: X\rightarrow PL_i$ are embeddings. Then
there is a Zariski open set $\bold U(k)$ in $\bold
L(k)=L_1\times\dots\times L_k$ such that for each point  $\bold
f(k)= (f_1,\dots,f_k)\in \bold U(k)$ the variety defined in $X$ by
the system of equations $f_1=\dots=f_k=0$ is smooth and
irreducible.
\end{Th}

A proof of the Lefschetz theorem can be found in \cite[Theorem 8.18]{Hartshorne}

\begin{Th}[A version of Hodge Index Theorem] \label{7.1}
Let  $X$ be a smooth (quasi) affine irreducible surface  and let
$L_1, L_2\in K(X)$ be very ample subspaces. Then we have
$[L_1,L_2]^2\geq [L_1,L_1][L_2,L_2]$.
\end{Th}

In the section 17 we give a proof of Theorem \ref{7.1} using
only the isoperimetric inequality for planar convex bodies and
Hilbert theory for degree of subvarieties in a projective space.

\begin{Th}[Algebraic analogue of Alexandrov--Fenchel
inequality] \label{7.2}
Let $X$ be an irreducible  smooth  $n$-dimensional (quasi) affine
variety and let $L_1,\dots,L_n\in K(X)$ be very ample subspaces. Then the
following inequality holds
$$[L_1,L_2, L_3\dots,L_n]^2\geq [L_1,L_1,
L_3\dots,L_n][L_2,L_2,L_3\dots,L_n].
$$
\end{Th}
\begin{proof} Consider $n$-tuples $(L_1, L_2, L_3,\dots,L_n)$,
$(L_1,L_1,L_3,\dots,L_n)$ and $(L_2,L_2,L_3,\dots,L_n)$ of elements
of the semigroup $K(X)$. According to Lefschetz theorem and
Theorem \ref{3.5} there is an $(n-2)$-tuple of  functions $f_3\in L_3,\dots,
f_n\in L_n$ such that the system $f_3=\dots+f_n=0$ is non-degenerate,
and defines an irreducible surface  $Y\subset X$, for which the
following equalities hold
$$[L_1,L_2,L_3, \dots,L_n]=[L_1, L_2]_Y,$$ $$[L_1,
L_1,L_3,\dots,L_n]=[L_1, L_1]_Y,$$
$$[L_2,L_2,L_3,\dots,L_n]=[L_2, L_2]_Y.$$  By Theorem \ref{7.1},
$$[L_1,L_2]^2_Y\geq [L_1,L_1]_Y[L_1,L_2]_Y,$$ which proves the theorem.
\end{proof}

\section{Semi-groups of integral points and convex bodies}
\subsection{Convex bodies and their stretch ratio}
\label{subsec-8}
One may expect that the number of integral points in a convex body
$\Delta\subset {\Bbb R}^n$ with large enough volume has the same order
of magnitude as its volume. The following example show that it is
not always true.

\begin{Ex} Define a convex body $\Delta \subset{\Bbb R}^n$ by the following inequalities: $ 1/2\leq x_1 \leq 3/4$, $0\leq
x_2\leq a, \dots, 0\leq x_n\leq a$. There is no integral point in $
\Delta$. But the volume of $\Delta$ equals to $(1/4) a^{n-1}$ and
can be as big as one wishes.
\end{Ex}

In this section we will define the stretch ratio of a convex body and
discuss its properties. In the next section  we will show that if
the stretch ratio of a sequence of convex bodies is bounded from above
and if their volumes tend to infinity then the number of integral
points in a convex body in this sequence is asymptotically equal to the
its volume.

We will measure volume in ${\Bbb R}^n$ with respect to the standard Euclidian metric.
Let $\Delta\subset {\Bbb R}^n$ be a bounded
$n$-dimensional convex body. Let $D$  be its diameter and $R$ the
radius of a largest ball which can be inscribed in $\Delta$.
In this section $B \subset {\Bbb R}^n$ will denote the
unite ball centered at the origin.

\begin{Def}
The {\it stretch ratio} of a convex body
$\Delta \subset {\Bbb R}^n$ is $D/R$ and will be denoted by $\mu(\Delta)$.
\end{Def}

For any $r \geq 0$ let $\Delta_r$ be the set of points $a$,
such that a ball of the radius  $r$ centered in  $a$ is contained
in $\Delta$ (in other words  $\Delta_r$ consists of points inside
$\Delta$, for which the distant to the boundary of $\Delta$ is
bigger than or equal to $r$).

\begin{Prop} \label{8.1}
1) For $0\leq r\leq R$ the set
$\Delta_r$ is non-empty and convex. 2) For every $b\in \Delta$,
there is a point $a\in \Delta_r$, such that the distance from $a$ to
$b$ is not bigger than $r \cdot \mu(\Delta) $.
\end{Prop}
\begin{proof} 1) For $0\leq r\leq R$  the set $\Delta_r$ contains
the center $O$ of the largest ball inscribed in $\Delta$ and so is
non-empty. Let $a_1,a_2\in \Delta_r$. The set $\Delta$
contains balls  $a_1+ rB$ and $a_2 +rB$. Because the set  $\Delta$ is
convex, it has to contain the ball $ta_1 +(1-t)a_2 +rB$ for $0\leq
t\leq 1$. So the body $\Delta_r$ contains the segment  $ta_1 +(1-t)a_2$
which proves that $\Delta_r$ is convex.
2) Take a point $b \in \Delta$ and a ball of radius  $R$ centered at $O$
which lies in $\Delta$. One easily sees that
for each $0 \leq \lambda \leq 1$ the ball of radius $\lambda R$ centered
at the point $b -\lambda (O-b)$
also belongs to the convex body $\Delta$. Plugging $\lambda =r/R$,
and noting that the length of the vector $(O-b)$ is
smaller than the diameter $D$ of the body $\Delta$, we get the required result.
\end{proof}

The body $\Delta_r$ (constructed out of $\Delta$) behaves
well with respect to the
Minkowski sum of convex bodies in the following sense. Let $\Delta_1, \Delta_2$ be convex bodies, let
$R_1$, $R_2$ be the biggest radii of balls which could be
inscribed in those bodies respectively and let $\Delta=\Delta_1 +\Delta_2$.

\begin{Cor} \label{8.2}
For $r_1 \leq R_1$, $r_2 \leq R_2$
we have the following inclusions:
$$(\Delta_{1, r_1} +\Delta_{2, r_2})+ (r_1+r_2) B\subseteq \Delta
\subseteq (\Delta_{1, r_1} +\Delta_{2, r_2}) + (r_1\mu (\Delta_1) + r_2
\mu (\Delta_2))B.$$
\end{Cor}
\begin{proof} We know that
$$ \Delta_{1, r_1} +r_1B \subseteq \Delta_1\subseteq \Delta_{1, r_1} +r_1\mu
(\Delta_1)B,$$ $$ \Delta_{2, r_2} +r_2 B \subseteq
\Delta_2\subseteq \Delta_{2, r_2} +r_2\mu (\Delta_2)B.$$ To get the claim it is enough to
sum up the above inclusions.
\end{proof}

\begin{Cor} \label{8.3}
With notations as in Corollary \ref{8.2}, the
set $(\Delta_{1, r_1} +\Delta_{2, r_2})$ contains the set
$\Delta_{\rho}$, where $\rho = (r_1\mu (\Delta_1) +r_2 \mu
(\Delta_2))$.
\end{Cor}

\begin{proof} Follows from the inclusion $$\Delta \subseteq
(\Delta_{1, r_1} +\Delta_{2, r_2}) + (r_1\mu (\Delta_1) + r_2\mu
(\Delta_2))B.$$
\end{proof}

 We will need an estimate of the volume of the set $\Delta \setminus
\Delta_r$. (We do not assume that the set $\Delta _r$ is not
empty).

\begin{Th} \label{8.4}
Given $r\geq 0$,
for every bounded $n$-dimensional convex
body  $\Delta\subset {\Bbb R}^n$ the volume of the set $\Delta
\setminus \Delta_r$, is not bigger than $(n-1)$-dimensional volume
$V_{n-1}(\partial \Delta)$ of the boundary $\partial \Delta$,
multiplied by $r$.
\end{Th}
\begin{proof} We will assume that the body $\Delta$ has the smooth
boundary $\partial \Delta$. This assumption is not restricted
because each bounded convex body could be approximated by convex
bodies with smooth boundaries.   At each point $x\in
\partial\Delta$ we fix a unite normal vector $\bold
n_x$ looking out of the domain $\Delta$. Consider the Riemannian
manifold  $\partial \Delta \times R$ --- the product of the manifold
$\partial \Delta$ equipped with the metric induced from ${\Bbb R}^n$
and the  line ${\Bbb R}$. Consider
the map  $F:\partial \Delta \times R\rightarrow {\Bbb R}^n$
defined by $(x, t) \mapsto x+t\bold n_{x}$. Let
$R_1(x)\leq \dots\leq R_{n-1}(x)$ be the  radii of the curvature
of the hyper surface $\partial \Delta$ at
the point $x$. It is easy to compute that the Jacobian $J(x,t)$
of the map $F$ at the point $(x,t)$ is equal to $(R_1(x)-t)\dots
(R_{n-1}(x)-t)/R_1(x)\dots R_{n-1}(x)$. We call
the domain $U = \{(x,t) \mid 0\leq t < R_i(x),~ i=1,\ldots,n-1 \} \subset
\partial \Delta \times R$, the {\it regular strip}.
At the points of the regular strip the Jacobian
$J$ ia positive and does not exceeded $1$. Let $\Sigma\subset
\Delta$ be the set of critical values of $F$. Let us show
that each point in the set $\Delta \setminus\Sigma$ is an image,
under the map $F$, of a point from the regular strip $U$. For
each point $a\in \Delta$, the minimum $t(a)$ of the distance of
$a$ to the boundary $\partial \Delta$ is
attained at some point $x(a) \in \partial \Delta$. The point $a$
could be represented in the form $a= x(a)+t(a)\bold n_{x(a)}$ where $ 0
<t(a)< R_1(x(a))$, otherwise the point $x(a)$ is not a local minimum
for the distance of $a$ the boundary. Thus $a$ is the image of
$(x(a), t(a))\in U$ under the map  $F$. Denote by $U_r$
the subset in the regular strip $U$, defined by the inequalities
$0\leq t<\min (r, R_i(x))$, $i=1, \ldots, n-1$.
The above arguments show that each point in $\Delta\setminus \Delta _r$ which
is not a critical value of the map $F$, belongs to the image under
the map $F$ of the set $U_r$. The theorem now follows by observing
 1) the volume of the domain $U_r$ is not bigger than the number
$r V_{n-1}(\partial \Delta)$, 2) The Jacobian of the map $F$ in the
domain  $U_r$ is positive and does not exceed $1$ and 3)
the set  $\Sigma$ of critical values of the map $F$ has zero measure by Sard's theorem.
\end{proof}

\subsection{Integral points in a convex body and its stretch ratio}
\label{subsec-9}
Consider a convex body having a large enough volume and assume that
its stretch ratio is less than some given constant.
In this section we will show that  the number of
integral points in such a body is, approximately, equal to the volume
of the body, and an integral of a  polynomial $f$ over such a body
is, approximately, equal to the sum of values of the polynomial over
all integral points which belong to the convex body.

For each $a =(a_1,\dots,a_n)\in \Bbb Z^n$, consider the unit cube
$K_{a} = \{(x_1,\dots, x_n)\in {\Bbb R}^n \mid a_i\leq x_i<a_i+1,~ i=1, \ldots, n\}$.
These unit cubes partition ${\Bbb R}^n$.

\begin{Prop} \label{9.1}
Let $\Delta \subset {\Bbb R}^n$ be a bounded measurable
set. Let $N_1$ (respectively $N_2$) be the number of the sets $K_{\bold a}$ which lie in
$\Delta$ (respectively intersect  $\Delta$ but do not lie in $\Delta$). Then the
volume  $V(\Delta)$ and the number $\# (\Delta\bigcap\Bbb Z^n)$ of
integral points belonging to $\Delta$ satisfy the following
inequalities:
\begin{enumerate}
\item $N_1\leq V(\Delta)\leq N_1+N_2,$
\item $N_1\leq\# (\Delta\bigcap\Bbb Z^n)\leq N_1+N_2.$
\end{enumerate}
\end{Prop}
\begin{proof} For a finite subset $A\subset \Bbb
Z^n$ put $K_A=\bigcup_{a\in A}K_a$. The number of points in $A$
is equal to the volume of the set $K_A$ as well as the number of
integral points in it. Given a bounded measurable set $\Delta$
let $A_1 = \{a \in \Bbb Z^n \mid K_a\subseteq \Delta\}$ and $A_2 = \{a \in \Bbb Z^n \mid K_a \bigcap\Delta\neq \emptyset$ but $K_a \nsubseteq \Delta\}$.
Let $A= A_1\bigcup A_2$. The number of integral points in the sets
$K_{A_1} $, $K_{A_2}$ and $K_A$ are equal to  $N_1$,
$N_2$ and $N_1+ N_2$ respectively. By definition we have
$K_{A_1}\subseteq \Delta \subseteq K_A$.  The proposition now follows
because the volume and the number of integral points are monotone
with respect to inclusion.
\end{proof}

Let $f: {\Bbb R}^n \rightarrow {\Bbb R}$ be a $C^1$ function. For a
measurable set $\Delta$ we will denote the integral $\int_{\Delta} f(x)dx$
by $\int_{\Delta}f$ and the sum $\sum_{x\in \Delta\bigcap
\Bbb Z^n}f(x)$ by $\sum_{\Delta} f$.

\begin{Prop} \label{9.2}
Let $M(f,\Delta)$ (respectively $M (df,\Delta)$)
be the maximum of $|\nabla f|$ (respectively $|df|$) on $\Delta$. The following inequalities hold:

$$|\int_{K_{A_1}}f - \sum_{K_{A_1}}f| \leq n^{1/2} M (d
f,\Delta)N_1,$$

$$|\int_{\Delta}f - \sum_{\Delta}f|\leq 2 M(f,\Delta)N_2.$$
\end{Prop}
\begin{proof} The first inequality follows from Mean Value Theorem, that is,
for $x,y \in K_a$, $|f(x)-f(y)|$ does not exceed the
diameter of $K_a$ (=$n^{1/2}$) multiplied by the maximum of $|\nabla f|$.
Second inequality follows from the inequalities $|\int_{\Delta}f|\leq M(f,\Delta)N_2,$ and
$| \sum_{\Delta}f|\leq M(f,\Delta)N_2.$
\end{proof}

\begin{Cor} \label{9.3}
$|\int_{\Delta}f - \sum_{\Delta}f|\leq n^{1/2} M (d f,\Delta)N_1 + 2
M(f,\Delta)N_2.$
\end{Cor}

\begin{Prop} \label{9.4}
Let $\Delta\in {\Bbb R}^n$ be a convex body
contained in a ball of the radius $D$. Then the number $N_2$ of
the sets $K_{\bold a}$ which intersect $\Delta$ but do not belong
to $\Delta$ satisfy the inequality
$$N_2\leq N_2(D,n)=2n^{1/2}\omega(n-1)(D+n^{1/2})^{n-1},$$
where $\omega(n-1)$ is the $(n-1)$-dimensional volume of the unite
$(n-1)$-dimensional sphere.
\end{Prop}

\begin{proof} As above let $K_{A_2}$ be the union of the sets $K_a$, which intersect $\Delta$
but not lie in it. Because the diameter
of the unite cube is $n^{1/2}$ we have
$\Delta_r\subseteq K_{A_2} \subseteq \Delta +rB$ where
$r=n^{1/2}$. According Theorem \ref{8.4} the volume of the set
$(\Delta +rB)\setminus \Delta_r $ does not exceed  the number
$V_{n-1}(\Delta +rB)2r$. The convex body $\Delta +rB$ is contained in a ball of
the radius  $D+r$ and hence $V_{n-1}(\Delta +rB)\leq
\omega_{(n-1)}(D+r)^{n-1}$ (note that if a convex body
$\Delta_1$  is contained in another convex body $\Delta_2$ then
$V_{n-1}(\partial \Delta_1)<V_{n-1}(\partial \Delta_2)$). Thus
$N_2=V_n(K_{A_2})\leq N_2(D,n)$.
\end{proof}
\begin{Prop} \label{9.5}
Let $\Delta$ be a convex body with
diameter is  $D$ which contains a ball of radius
$R$. Then the volume of $\Delta$ is bigger than or equal to
$$V(D,R,n)=D R^{n-1} \Omega_{n-1}/2(n-1)!$$, where $\Omega_{n-1}$ is the volume of the
unite $(n-1)$-dimensional ball.
\end{Prop}
\begin{proof} Let $O$ be the center of a ball of radius $R$ contained in $\Delta$.
Since the diameter of $\Delta$ is $D$, there is
a point $b \in \Delta$ whose distance from $O$ is bigger than or equal to $D/2$.
Then $\Delta$ contains the cone of revolution whose apex is $b$, its base is an
$(n-1)$-dimensional ball of radius $R$ centered at $O$
and its height equal to $D/2$. The volume of this cone is $V(D,R,n)$.
\end{proof}
\begin{Cor} \label{9.6}
Let  $\Delta \subset {\Bbb R}^n$ be a convex
body with volume $V(\Delta)$ and the stretch ratio  $\mu(\Delta)$.
If $\Delta$ contains a unite ball then $$ N_2/V(\Delta)\leq
F(\mu(\Delta),n)V(\Delta)^{-1/n},$$ for an explicitly defined function $F$.
\end{Cor}
\begin{proof} Using the estimates in Propositions 9.4 and 9.5
one see that, up to explicitly computable constants, the quantity
$N_2/V(\Delta)$ can be estimated from above by the
expression  $$(D +n^{1/2})^{n-1}/R^{n-1} D. $$ Using the relations
$R>1$, $V^{1/n}(\Delta)< D$ and $\mu(\Delta)= D/R $ one obtains
$$(N_2/V(\Delta) \leq \mu(\Delta) +n^{1/2})^{n-1} V(\Delta)^{-1/n}.$$
\end{proof}

\begin{Rem}
Since $D>V^{1/n}(\Delta)$ we have $R>
V^{1/n}(\Delta)/\mu(\Delta)$. So if the volume of $\Delta$ is
bigger than  $\mu(\Delta)^n$, then $\Delta$ automatically contains a
unite ball and we can drop the condition of containing a unit ball in Corollary \ref{9.6}
for such convex bodies.
\end{Rem}

Let $\Delta$ be a bounded convex $n$-dimensional body. Denote the
multiplication of $\Delta$ by a scalar $\lambda > 0$
with $\lambda \Delta$. The following relations hold

$$V(\lambda\Delta)= \lambda ^{n} V(\Delta),$$
$$V_{n-1}(\partial(\lambda\Delta))= \lambda ^{n-1}
V_{n-1}(\partial(\Delta)).$$

Let $f:{\Bbb R}^n\rightarrow R$ be a homogeneous $C^1$ function  of
degree $\alpha\geq 0$, i.e. $f(\lambda x)=\lambda^{\alpha}f(x)$.
From homogeneity of $f$ we have:

$$M(f,\lambda \Delta)=\lambda ^{\alpha}M(f,\Delta),$$
$$M(df,\lambda \Delta)=\lambda ^{\alpha-1}M(f,\Delta),$$
$$\int_{\lambda \Delta}f(x)dx=\lambda ^{\alpha+n}\int_{\Delta}f(x)dx.$$

\begin{Th} \label{9.7}
Let  $\Delta \subset {\Bbb R}^n$ be a bounded
$n$-dimensional convex body and let  $f:{\Bbb R}^n \rightarrow {\Bbb R}$ be a
homogeneous $C^1$ function of the degree $\alpha\geq 0$. Then
$$ \lim_{\lambda\rightarrow \infty}\frac{\sum _{x\in \lambda
\Delta\cap\Bbb Z^n }f(x)}
{\lambda^{\alpha+n}}=\int_{\Delta}f(x)dx.$$
\end{Th}

\begin{proof} From Corollary \ref{9.3} we have
$$\frac{|\int_{\lambda \Delta}f(x)dx -\sum _{x\in \lambda
\Delta\cap\Bbb Z^n }f(x)|}{\lambda^{\alpha+n}}\leq\frac { n^{1/2} M
(d f,\lambda \Delta)}{\lambda^{\alpha}} \cdot \frac {N_1(\lambda
\Delta)}{\lambda^n} +\frac{ 2 M(f,\lambda
\Delta)}{\lambda^\alpha}\cdot \frac{N_2(\lambda \Delta)}{\lambda^n}.$$

As $\lambda\rightarrow \infty$, the expressions $\frac {N_1(\lambda
\Delta)}{\lambda^n}\leq V(\Delta)$ and $\frac{ 2 M(f,\lambda
\Delta)}{\lambda^\alpha}=2M(f,\Delta)$ remain bounded but $\frac {
n^{1/2} M (d f,\lambda \Delta)}{\lambda^{\alpha}}$ tends to $0$ (if
$\alpha =0$ the function $f$ is constant and the last term vanishes)
and $\frac{N_2(\lambda \Delta)}{\lambda^n}\rightarrow 0$ (see
Corollary \ref{9.6}). This proves the theorem.
\end{proof}

Let $f:{\Bbb R}^n\rightarrow {\Bbb R}$ be a  polynomial of degree $k$  and let
$f=f_0+f_1+\dots+f_k$ be its decomposition into homogeneous
components.

\begin{Cor} \label{9.8}
$$ \lim_{\lambda\rightarrow \infty}\frac {\sum _{x\in \lambda
\Delta\cap\Bbb Z^n }f(x) }{\lambda^{n+k}} =\int_{\Delta}f_k(x)dx.$$
\end{Cor}

\begin{proof} According to Theorem \ref{9.7}, for any $0\leq i\leq k$
we have
$$\lim_{\lambda\rightarrow \infty}\frac {\sum _{x\in \lambda
\Delta\cap\Bbb Z^n }f_i(x) }{\lambda^{n+i}}
=\int_{\Delta}f_i(x)dx.$$ The corollary easily follows from this.
\end{proof}
\begin{Cor} \label{9.9}
Let $\Delta$ be a bounded convex body. Then
$$ \lim_{\lambda\rightarrow \infty}\frac
{\# (\lambda \Delta\bigcap\Bbb Z^n)}{\lambda^n}= V(\Delta).$$
\end{Cor}
\begin{proof} Apply Theorem \ref{9.7} to $f\equiv 1$.
\end{proof}

\subsection{Semigroups of integral points}
\label{subsec-10}
The set of points $(h, \bold x)$ in ${\Bbb R}\times{\Bbb R}^n$
with $h \geq 0$ is called the {\it positive half-space}.
We call a semigroup $G\subset \Bbb Z\times \Bbb Z^n$ a {\it graded
semigroup}, if the following conditions are satisfied: 1) $G$ is
contained in the positive half-space.
2) For each  $d \in \Bbb N$ the set of elements of $G$ of the form  $(d, \bold m)$ is non-empty.
We will say that an element $(d, \bold m) \in G$ has {\it degree} $d$.

Now we define the class of semigroups $G$ which will play a key
role in for us. We need the following two definitions:

1)  A closed convex $(n+1)$-dimensional  cone $C$ in the positive
half-space  is called a {\it  positive cone} if its intersection
with the horizontal hyperplane $h=0$ contains only the origin.

2) To an integral point $A=(1,\bold x)$,
with the first coordinate  equal to $1$, and a subgroup $T \subset \Bbb Z^n$,
we associate  the subgroup  $lA+ T
\subset \Bbb Z \times \Bbb Z^n$ of vectors $lA+\bold x$, where
$l\in \Bbb Z$, $\bold x\in T$. Obviously if $A_1-A_2\in T$, the subgroups
$lA_1+ T$ and $lA_2+ T$ coincide.

\begin{Def} Let $C$ be a positive cone, $T$ a
subgroup of a finite index in $\Bbb Z^n$, and $A$
an integral point with the first coordinate equal to $1$ (the point
$A$ is defined up to addition of an  element  from the group $T$). We say that a
semigroup $G\subset  \Bbb Z\times\Bbb Z^n$ has type $(C,T,A)$
if $G = C \cap lA\times T$.
\end{Def}

The following statement is clear.

\begin{Prop} \label{10.1}
Two semigroups of types  $(C_1,T_1,
A_1)$ and $(C_2,T_2, A_2)$ coincide if and only if $C_1=C_2$, $T_1=T_2$ and
the difference $A_1-A_2$ belongs to the group $T_1=T_2$.
\end{Prop}

And a few extra definitions. Let $G$ be a graded semi-group.
\begin{enumerate}
\item $G$ has {\it finite sections}, if for every $d>0$ the set of elements
of degree $d$ in $G$ is finite. We will denote the number of elements of degree
$d$ by $H_G(d)$. We call $H_G$ the {\it Hilbert function} of a graded semigroup
with the finite sections.

\item $G$ has {\it conic type}, if it is contained in
a positive cone.

\item $G$ has {\it limited growth}, if it has finite sections and
$H_G(d)<q d^n$, for a constant $q$.

\item $G$ has {\it complete rank}, if the subgroup  $\Bbb Z \times \Bbb
Z^n $, generated by the semigroup $G$,  has finite index in $\Bbb Z \times \Bbb Z^n $.

\item $G$ is {\it saturated}, if the subgroup
generated by the semigroup $G$ is the whole $\Bbb Z
\times \Bbb Z^n $.
\end{enumerate}

It is clear that if a semi-group has conic type then it has a limited
growth. Also if it is saturated, then it has complete rank.

We are interested  in semigroups with finite sections, complete rank
and limited growth. We will see that one can obtain a more or less
complete description  asymptotic  behavior  of such semigroups.

Suppose a semigroup $G$ is contained in a  semigroup $M$
of the type $(C,T,A)$. Denote the sections of
$G$, $M$ and the cone $C$ by the hyperplane $h=d$ respectively by $G(d)$, $M(d)$ and $C(d)$.
Consider the function $r_{G,M}(d)$ defined as the minimum distance
from a point in the set $M(d)\setminus G(d)$ to the boundary of the section $C(d)$.
\begin{Def} A semigroup $M\subset  \Bbb Z\times\Bbb Z^n$ of a type
$(C,T,A)$ {\it approximates the semigroup $G$} if:
\begin{enumerate}
\item The semigroup $M$ contains the semigroup $G$.
\item We have $$ \lim _{d\rightarrow \infty}\frac
{r_{G,M}(d)}{d} =0$$
\end{enumerate}
\end{Def}
The condition 2) is  equivalent  to the following condition: 2'):
there exists a  function $P(\rho)$ such
that for all $d \Bbb N$ and $\rho > 0$ we have $r_{G,M}(d) <\rho d + P(\rho)$.

The following statement is clear.
\begin {Prop} \label{10.2}
Let $G_1\subseteq G_2\subset \Bbb
Z\times \Bbb Z^n $ be semigroups which  are approximated by the
semigroups $M_1,M_2$ of the types $(C_1,T_1,A_1)$ and $(C_2,T_2,A_2)$ respectively.
Then $C_1\subseteq C_2$, $\Delta(G_1)\subseteq \Delta (G_2)$ and
$T_1\subseteq T_2$. The point $A_2$ can be chosen equal to $A_1$.
\end{Prop}

The following Theorem \ref{10.3} will be important for us. Let $A$ be a finite
subset in $\Bbb Z^n\subset {\Bbb R}^n $, and let $T$ be the
subgroup of $\Bbb Z^n$ generated by $A$. Denote the convex hull of $A$ by $\Delta \subset {\Bbb R}^n$.
Let $k*A$ be the set $\underbrace{A+\dots +A}_{k \textup{ times}}$, which
consists of all sums of $k$-tuples $a_1,\dots, a_k$ of elements
of the set $A$.

\begin{Th}[\cite{Khovanskii-finite-sets}] \label{10.3}
Let $T \subset \Bbb Z^n$ be a
subgroup of finite index. Then there is a constant $P$
(independent on $k$) such that every point in $k\Delta \cap T$ whose distance to the
boundary $\partial (k\Delta)$ of the polyhedron is not smaller than
$P$ belongs to $k*A$.
\end{Th}

\begin{Th} \label{10.4}
Let $G\subset \Bbb Z\times \Bbb Z^n $ be a
semigroup of complete rank and with finite sections. If $G$
has limited growth then it is of conic type.
\end{Th}
\begin{proof} The idea of the proof is as follows. If
$G$ does not belong to any positive cone, then given any constant $L$ one can find
a certain sequence $d_i$ of natural numbers such that the Hilbert  function $H_G(d_i)$ is bigger than
$Ld_i^n$. Proof is based on Theorem \ref{10.3}. To start we need some auxiliary constructions.

Let $A$ be any point in $G(1)$ (recall that
for any $i \in \Bbb N$, $G(i)$ is non-empty).
Consider the sets $G(d)- dA$ considered as subsets of the coordinate hyperplane
$h=0$. They possess the following properties: 1) The origin belongs to each set
$G(d)- dA$, 2) $(G(d_1)-d_1A) + (G(d_2)-d_2A) \subseteq G(d_1+d_2)- (d_1+d_2)A$,
3) If $d_1\leq d_2$ then $(G(d_1)-d_1A) \subseteq (G(d_2)-d_2 A)$, 4) The union of sets
$G(d)-dA$ generates a subgroup $T$ of finite index in
$\Bbb Z^n$, 5) For $k \gg 0$ the set $G(k)-kA$ generates the
subgroup $T$ (because the subgroup $T$ is finitely generated).

Denote by $\Delta (d)$ the convex hull of the set $G(d)$. For $k\gg
0$ the set  $G(k)$ generates the group  $T$ which has
complete rank and hence the polyhedron $\Delta(k)$ contains $n$ linearly independent vectors.
Fix a $k_0$ for which this is the case. Consider the polyhedron $\frac {1}{k_0}\Delta (k_0)$
located in the plane $h=1$.
One can find an $n$ dimensional ball of a radius $R>0$ in this polyhedron. Let $(1, \tilde O)$ be the center
of this ball. Let us show that for each point $(m,\bold x)\in G$ the distance $l$
from the points $(1,\frac {\bold x}{m})$ to $(1,\tilde O)$ can be
estimated from above. In fact, the $n$ dimensional volume $V$ of a
convex body, which contains a ball of the radius  $R$ and a  point
whose distance to the center of the ball is equal to  $l$ has to be
bigger than or equal to $V=cR^{n-1}l/(n-1)!$, where $c$ is the volume of
the unite  $(n-1)$ dimensional ball. In the next paragraph we will
show that there is a sequence $d_i$ of  arguments, such that the
limit of $H_G(d_i)/d_i^n$, as $i \to \infty$,
is bigger than or equal to $V/I$ where $I$ is the index of the
semigroup $T$ in $\Bbb Z^n$. From the assumption $H(G,d_i)<Ld_i^n$ we see that $V/I\leq L$
and that $l\leq l_0= I L(n-1)/ c(n-1)R^{n-1}$ which give an estimate of
the distance  $l$. Now the semigroup $G$ belongs to the positive cone $C$,
whose section  $C(1)$, by the hyperplane $h=1$, is the  ball of the
radius  $l_0$ centered at $(1,\tilde O)$. This shows that $G$ has conic type.

Now let us show how to construct the sequence $\{d_i\}$.
Take $k_0$ and $(m, \bold x)$ as above. Then the convex hull of the section $G(k_0)$
projected to the plane $h=1$ has volume bigger than or equal to $V$.
Let $d=k_0m$. The convex hull of the section $G(d)$ contains
$d\Delta$ whose volume is greater than or equal to $Vd^k$. The points in the section
$G(d)$ generate a subgroup  $T\in \Bbb Z^n$ of the index  $I$. For $i \in \Bbb N$
put $d_i = id = ik_0m$. Now applying Theorem \ref{10.3} to the set $G(d)$
we see that, for large enough $i$, $H( G,d_i)/d_i^n$ can not be smaller than $V/I$.
This finishes the proof of the theorem.
\end{proof}

\begin{Th} \label{10.5}
Let $G\subset \Bbb Z\times \Bbb Z^n $ be a
semigroup of complete rank, with finite sections. If the semigroup
has a limited growth then the semigroup $G$ then there is a
semigroup $M$ of the  type $(C,T,A)$, which approximates the
semigroup  $G$. Such semigroup $M$ is unique.
\end{Th}
\begin{proof} The proof of Theorem \ref{10.5} resembles the proof of
Theorem \ref{10.4}. Both of them are based on Theorem \ref{10.3}.
Denote by  $G(\leq d)$ the finite subset in the semigroup $G$
consisting of all the elements with degrees not bigger than $d$. Let
$\tilde G(\leq d)$ be the projection of the set $G(\leq d)$ from
the origin to the hyperplane $h=1$, (i.e. if $(m,\bold x)\in G(\leq d)$,
then $(1,\frac {\bold x}{m})\in\tilde G(\leq d)$).
Let $\tilde \Delta(\leq d)$ be the convex hall of $\tilde
G(\leq d)$.  In Theorem \ref{10.4} we obtained a increasing sequence of the convex bodies
$\tilde \Delta(\leq 1)\subseteq \dots,\subseteq \tilde \Delta(\leq d)\subseteq\dots $
all contained in a bounded convex body. Let $\Delta = \bigcup_{1\leq q}
\Delta(\leq d)$ and $\bar \Delta$ be its closure.
We will show that the semigroup $M$ of the type $(C,T, A)$
approximates the semigroup $G$, where:

$C =$ the positive cone, whose section by the hyperplane $h=1$
coincides with $\bar \Delta$.

$T =$ the intersection of the subgroup in $\Bbb Z\times \Bbb Z^n$,
generated by the semigroup $G$ with the group $\Bbb Z^n=\{0\}\times
\Bbb Z^n$,

$A =$ any element of degree one in the semigroup  $G$, $A\in
G(1)$.

Fix a $\rho >0$ and a positive cone
$C_{\rho}\subset C$ such that its section  $C_{\rho}(1)$ by the
hyperplane $h=1$ lies strictly inside the section $C(1)$ of $C$,
and such that the distance from $C_{\rho}(1)$ to the boundary
of $C(1)$ is greater than $\rho$. To proof the theorem
it is enough to show that given the cone $C_{\rho}$ there is a
constant  $P(\rho)$ (independent on $k$) such that any point inside
the section $C_{\rho}(k)$ whose distance to boundary $\partial
(C_{\rho}(k))$ is bigger than or equal to $P_{\rho}$ and which is
representable in the form $kA +T$, belongs to the semigroup $G$.

For $K > k$ regard $G(k)$ as a subset of $G(K)$ by adding the vector
$(K-k)A \in G$ to all the points in $G(k)$.
Fix any (small) positive number $\rho$. The increasing sequence of
the convex bodies  $\tilde \Delta(\leq d)$ converges to the body
$\bar \Delta$. So starting from some number $d_1$ the  Hausdorff
distance between the bodies $\tilde \Delta(\leq d)$ and $\bar
\Delta$ is smaller than $\rho$. Then starting from some number $d_2$ the
set $G(\leq d_2)$ and the semigroup $G$ generate the same subgroup.
Let $m \in \Bbb N$ be bigger than $d_1$ and $d_2$.
Consider the section $G(m!)$ of the semigroup $G$. It has the
following properties:

1) The projection of the section $G(m!)$ from the origin to the
hyperplane $h=1$ contains the set $\tilde G(\leq q_0)$. In fact if
$(p,\bold x)\in G(\leq q_0)$ then $ \frac {m!}{p}(p, \bold x)\in
G(m!)$ because $G$ is a semigroup and the number  $m!$ is
divisible by the number $p\leq m$.

2) The differences of the points in the section $G(m!)$ generate the
group $T$. Because the points of the section $G(\leq d_2)$ could be
shifted to the section $G(m!)$ by adding the vector $kA$ for an
appropriate $k$. By the assumption the
intersection of the group generated by the set $G(\leq d_2)$ with
the horizontal hyperplane is equal to $T$. So the differences of
the points on in $G(m!)$ generates the group $T$.

Now let us apply Theorem \ref{10.3} to the section $G(m!)$  and the sums
$$\underbrace{G(m!)+\dots+G(m!)}_{k \textup{ times}},$$
which belong to $G(k m!)$. Let $T(A, P, km!)$ be the subset
of the group $lA +T$ consisting of the points in the set
$km! \tilde \Delta (\leq m!)$ such that their distance to boundary
of this polyhedra is bigger than $P$. According to Theorem \ref{10.3}
there is  a constant $P$, such that for each  $k$ any point in the
set  $T(A,P, km!)$ belongs to the semigroup $G$.
Thus we may find many points from the group $lA+T$ in the sections $G(d)$
of the semigroup $G$ where $d$ is divisible by $m!$. Now let
$d$ be equal to $km! +q$ with $0\leq q <m!$. The section
$G(d)$ contains  points of the set $T(A, P, km!)+ qA$ which belong
to the group $lA+ T$. Denote by $D$ the diameter of the polyhedron
$\tilde \Delta (\leq m!)$.  We show that all the points of
$lA+ T$ in the polyhedron $(km! +q)\tilde \Delta (\leq
m!)$ such that their distance to the boundary of this polyhedron is bigger
than  $m!D+P$ are in the semigroup $G$. Indeed, such points are
inside the polyhedron $km! \tilde \Delta (\leq m!) + qA $ and their
distance to the boundary is bigger than or equal to $P$. So they are in the semigroup $G$.
We proved that each point in $lA+T$ which belongs to the
section $C(h)$ and whose distance to the boundary is bigger than
$\rho h +m!D+P$ belongs to $G$. The theorem is now proved.
\end{proof}
\begin{Def} Consider a semigroup $G$
which can be approximated by a semigroup $M$ of the type
$(C,T,A)$. The {\it Newton convex body $\Delta(G)$} of the semigroup $G$
is the convex body obtained by the intersecting the cone
$C$ with the  hyperplane $h=1$ in the space ${\Bbb R}\times {\Bbb R}^n$.
We regard $\Delta(G)$ as a subset of ${\Bbb R}^n$.
\end{Def}
\begin{Th} \label{10.6}
Assume that a semigroup $G$ can be
approximate by a semigroup of the type $(C,T,A)$. Let $f:{\Bbb R}^n\rightarrow {\Bbb R}$ be a $C^1$ homogenous function of degree
$\alpha\geq 0$. Then
$$ \lim_{d\rightarrow \infty}\frac{\sum _{x\in G(d)}f(x)} {d^{\alpha+n}}
=\frac{1}{ind (T)}\int_{\Delta(G)}f(x)dx,$$ where $ind (T)$ is the
index of subgroup $T\subset \Bbb Z^n$,  $G(d)$ is the set of
elements of degree $d$ in $G$ and $\Delta (G)$ is the Newton convex body of the semigroup $G$.
\end{Th}

\begin{proof} If the index of the subgroup $T\in \Bbb Z^n$ is equal
to $1$, the theorem follows from Theorem \ref{9.7}. If the index is
bigger than $1$, one can make a linear change of variables and
transform the subgroup  $T$ into the whole lattice $\Bbb Z^n$. Such change
of variable changes the volume by the factor $1/ind (T)$.
Also $1/ind(T)$ is responsible for the asymptotical behavior of
the sum of values of $f$ on the degree $d$ elements of the
semigroup as $d \to \infty$.
\end{proof}

\begin{Cor} \label{10.7}
Assume that a semigroup $G$ can be
approximate by a semigroup of the type $(C,T,A)$. Let $f:{\Bbb R}^n\rightarrow {\Bbb R}$ be a polynomial and let $f=f_0+f_1+\dots+f_k$ be
its decomposition into homogenous components.  Then $$ \lim_{d\rightarrow \infty}\frac {\sum
_{x\in G(d)}f(x) }{d^{n+k}} = \frac{1}{ind (T)}
\int_{\Delta(G)}f_k(x)dx.$$
\end{Cor}

\begin{Cor} \label{10.8}
Assume that a semigroup $G$ can be
approximated by a semigroup of the type $(C,T,A)$. Then the Hilbert
function  $H_G$ has the following asymptotical behavior:
$$\lim_{d\rightarrow \infty}\frac {H_G(d)}{d^n}=\frac{1}{ind (T)} V(\Delta(G)),$$
where $V(\Delta (G))$ is the $n$-dimensional volume of the Newton convex body
$\Delta (G)$ of the semigroup $G$.
\end{Cor}

Assume that a graded semigroup $G$ is contained in another graded
semigroup $G_1\subset \Bbb Z \times \Bbb Z^n$ of complete rank
and with limited growth. For such semigroups all properties we
are interested in are corollaries of the results proved above.
Let us discuss this in more details: let $A \in G(1)$ be a degree $1$ element in $G$.
Denote by $T$ the intersection of the subgroup
generated by $G$ and $\Bbb Z^n=\{0\}\times
\Bbb Z^n$. Assume that the group $T$ has rank $k$. Consider
the subgroup $M \subset \Bbb Z^n=\{0\}\times \Bbb Z^n$
consisting of all the elements $m$ which after multiplication by a natural
number $l(m)$ lie in $T$ i.e. $l(m)m\in T$. The
group $M$ is  isomorphic to $\Bbb Z^k$ and after a choice
of a basis can be identified  with this group.
The group $T$ is a subgroup of a finite index in $M$. The
group generated by the semigroup $G$ is contained in the group $\langle A \rangle
\times M\simeq \Bbb Z\times \Bbb Z^k$ generated by $A$ and $\{0\} \times M$.
By the assumption the semigroup  $G$ is contained in a graded
semigroup of complete rank and with limited growth. So the
semigroup $G$ is contained in a closed positive cone $C$ (Theorem \ref{10.4}).
Let us call the vector space generated by the group $\langle A \rangle
\times M$ the {\it space of the semigroup $G$} and denote it by $V(G)$.
It is isomorphic to ${\Bbb R}\times {\Bbb R}^k$.
The group $\langle A \rangle \times M$ is a lattice in this
space and defines a Euclidean metric in the space of the semi-group $G$
in which the volume of the parallelepiped given by the generators of
$\langle A \rangle \times M$ is equal to $1$.
Semigroup $G$ belongs to the positive closed cone $C_1 =
C\cap {\Bbb R}\times {\Bbb R}^k$ in this space. Now one can apply Theorem
\ref{10.6} to the semigroup $G$.

\begin{Def} Let $C(G)$ denote the closure of the convex hull
of $G \cup \{0\}$. The intersection of $C(G)$ with the horizontal hyperplane
$h=1$, namely $\Delta(G)=C(G)\cap \{h=1\}$, will be called
The {\it Newton convex body of the semigroup $G$}.
\end{Def}

Let $G\in \Bbb Z\times \Bbb Z^n$ be a graded semigroup which is
contained in a graded semigroup $G_1\subset \Bbb Z \times \Bbb Z^n$
of complete rank and with limited growth. Assume that the rank of
the group  $T(G)$ is equal to $k$. We have the following:

\begin{Th} \label{10.9}
The Newton convex body $\Delta (G)$ of the
semigroup $G$ is a bounded $k$-dimensional convex body.  In the
space $V(G)$ of the semigroup $G$, $C(G)$ is a convex cone
of maximum dimension and $G$
can be approximated by a semigroup of type $(C(G),T,A)$ in $V(G)$.
Finally the Hilbert
function $H$ of the semigroup $G$ has the following asymptotical
behavior:  $$\lim \frac {H(d)}{d^k}= \frac {k! V_k(\Delta (G))}{ind
T(G)},$$ where $V_k$ is the $k$-dimensional volume and $ind
(T)$ is the index of the subgroup $T$ in $M$.
\end{Th}

\subsection{Addition of graded semigroups} \label{subsec-11}
Let us start with a lemma about sum of subgroups in the lattice $\Bbb Z^n$.
\begin{Lem} \label{11.1}
Let $T_1,T_2$ be subgroups of finite index in
$\Bbb Z^n\subset {\Bbb R}^n$ and let $T$ be the sum of these
subgroups. Then there exists a number  $N$ with the following
property: for any representation of an element $a\in T$ in the form
$a=\bold x +\bold y$ where $\bold x, \bold y\in {\Bbb R}^n$, one can
find elements $b\in T_1$ i $c\in T_2$ such that $a=b+c$ and
$||a-\bold x||<N$, $||b-\bold y||<N$.
\end{Lem}
\begin{proof} Fix generators
$p_1,\dots, p_n$ in the group  $T_1$. The vectors $p_1,\dots, p_n$ form
a basis in ${\Bbb R}^n$. Each vector $\bold u\in {\Bbb R}^n$ can be
represented in the form  $\lambda_1 p_1+\dots +\lambda_np_n= Q_1
+M_1$, with $Q_1=[\lambda_1]p_1+\dots+[\lambda_n]p_n$
belongs to  $T_1$, and $M_1= (\lambda_1
-[\lambda_1])p_1+\dots+(\lambda_n -[\lambda_n])p_n$ has a length not
bigger than  $D_1$ where $D_1=\sum||p_i||$. 2) In a similar way each
vector $\bold v\in {\Bbb R}^n$ can be represented in the form $Q_2
+M_2$, where the vector $Q_2$ belongs to the group $T_2$, and the
vector $M_2$ has a uniformly bounded length  $||M_2||< D_2$. For
each vector $r_i$  from a finite subset in the group $T_1+T_2$ of
vectors whose length is not bigger than $D_1+D_2$.  Fix a
representation in the form $r_i=b_i+c_i$ where $b_i\in T_1$ and
$c_i\in T_2$. Denote by $D_3$ the number $D_3=\sum ||b_i|| +\sum
||c_i||$.

We proceed with the proof as follows. Assume that an element $a\in T$ is
represented in the form $a=\bold x +\bold y$. For the vector $\bold x$
(respectively $\bold y$) one can find a vector $Q_1\in T_1$ (respectively $Q_2 \in T_2$)
such that $||\bold x -Q_1||<
M_1$ (respectively $||\bold y -Q_2||<
M_2$). For a vector $r=(a -Q_1 -Q_2)\in
T $ whose length is not bigger than $D_1+D_2$ there is a
representation in the form  $r=b_i +c_i$ where $b_i\in T_1, c_i\in
T_2$ and $||b_i||, ||c_i|| < D_3$.  So we have represented the vector
$a$ in the form $a=(Q_1 +b_i) +(Q_2+c_i)$, where the vectors $(Q_1
+b_i)$ and $(Q_2+c_i)$ belong to the groups $T_1$ and $T_2$ respectively
and the following inequalities holds $$|| \bold x-(Q_1 +b_i)||<D_1+D_3,$$ $$|| \bold
y-(Q_2 +c_i)||<D_2+D_3.$$ Taking
$N=D_1+D_2 +D_3$ completes the proof.
\end{proof}

Consider the following addition on ${\Bbb R} \times {\Bbb R}^n$ between
the vectors which have the same first coordinate:
$$(h, \bold x_1) \oplus_t (h, \bold x_2) = (h, \bold x_1 + \bold x_2).$$
Equivalently $(h, \bold x_1) \oplus_t (h, \bold x_2) =
(h, \bold x_1)+(h, \bold x_2)-h\bold e$ where $\bold e$ is the
unit vector $(1, 0, \ldots, 0) \in {\Bbb R} \times {\Bbb R}^n$.
For two subsets $A, B \subset {\Bbb R} \times {\Bbb R}^n$ let
$A \oplus_t B$ be the collection of all $a \oplus_t b$ where
$a \in A$, $b \in B$ and $a,b$ have the same first coordinates.

The following statement is obvious.
\begin{Prop} \label{12.2}
Let $G_1,G_2$ be graded semigroups then:
1) $G_1\oplus_t G_2$ is a graded semigroup. 2) If
$G_1, G_2$ have finite sections (respectively conic type) then the semigroup
$G_1\oplus _dG_2$ also has finite sections (respectively conic type).
\end{Prop}

Let $A_1=(1, \bold x_1)$, $A_2=(1, \bold x_2)$ be two points in
the hyperplane $\{h=1\}$. Put $A = A_1\oplus_t A_2= (1,\bold x_1+\bold
x_2)$.

\begin{Th} \label{11.3}
Let $G_1,G_2$ be graded semigroups which
can be approximated by semigroups $M_1$,$M_2$ of the types
$(C_1,T_1,A_1)$ and $(C_2,T_2,A_2)$. Then the semigroup $G_1\oplus_t
G_2$ can be approximated by a semigroup of the type  $(C,T,A)$
where $C=C_1\oplus_t C_2$,  $T=T_1+T_2$ and $A=A_1\oplus_t A_2$.
\end{Th}
\begin{proof} By assumption the semigroups  $G_1$ and $G_2$ belong to
the cones $C_1$ and $C_2$ respectively and $G_1\oplus_t G_2$ is contained
in the cone $C_1\oplus_tC_2$. Also $G_1$ and
$G_2$ contain points $A_1$ and $A_2$ and their intersections with
the hyperplane $h=0$ are equal to $T_1$ and to $T_2$
respectively. So the semigroup $G_1\oplus _t G_2$ contains the
point $A= A_1\oplus_t A_2)$ and its intersection with the hyperplane
$h=0$ is $T_1 + T_2$.
By the assumption the semigroups $M_1$ and $M_2$ approximate the
semigroups $G_1$ and $G_2$. So: 1) Each point in the section $G_1
(d)= G_1\cap \{h=d\}$ whose distance to the boundary $C_1(d)=
C_1\cap \{h=d\}$ of the section is bigger than $r_1=r(G_1,M_1)(d)$
has to belong to the subgroup $T_1$ shifted by the vector  $d A_1$.
2) Each point in the section $G_2 (d)= G_2\cap \{h=d\}$ whose
distance to the boundary $C_2(d)= C_2\cap \{h=d\}$ of the section is
bigger than $r_2=r(G_2,M_2)(d)$ has to belong to the subgroup $T_2$
shifted by the vector  $d A_2$.

Let us reformulate the  statement from the above paragraph. Let us
consider in the hyperplane $h=0$: 1) the convex body $\Delta_1(d)=
C_1(d) -dA_1$ and its  $(\Delta_1(d))_{r_1}$ where
$r_1=r(G_1,M_1)(d)$. Each point belonging to the intersection of the
set $(\Delta_1(d))_{r_1}$ and of the group $T_1$ belongs to the set
$G_1(d) -dA_1$. 2) the convex body $\Delta_2(d)= C_2(d) -dA_2$ and
its  $(\Delta_2(d))_{r_2}$ where $r_2=r(G_2,M_2)(d)$. Each point
belonging to the intersection of the set $(\Delta_2(d))_{r_2}$ and
of the group $T_2$ belongs to the set $G_2(d) -dA_2$.

Let $R_1 $ and $R_2$ be the radiuses of the maximal balls in the
convex bodies $\Delta_1=C_1(1)-\bold e$ and $\Delta_2=C_2(1)-\bold
e$ in the hyperplane $h=0$ and let $D_1$ and $D_2$ be the diameters
of those bodies. The bodies $\Delta_1$ and $\Delta_2$ are the
sections of the cones $C_1$ and $C_2$ by the hyperplane $h=1$
shifted by adding the vector $-\bold e$ to the hyperplane $h=0$. The
similar characteristics of the sections of the cones $C_1$ and $C_2$
by the hyperplane $h=t$ shifted by adding the vector $-t\bold e$ to
the hyperplane $h=0$ are correspondingly  $(tR_1,tD_1)$ and
$(tR_2,tD_2)$. For each cone the ratio $tD_i/tR_i$ is independent on
$t$.

Consider in the hyperplane $h=0$ the convex bodies $\Delta_1(d)=
C_1(d) -dA_1$ and $\Delta_2(d)= C_2(d) -dA_2$. According to the
corollary 8.3 for every $r\ll d$ the Minkowski sum of the sets
$(\Delta_1(d))_{r} $ and $(\Delta_2(d))_{r} $ contains the set
$(\Delta(d))_{Qr} $, gde $Q= (D_1/R_1) + (D_2/R_2)$.

Take $r(d)= \max[r_1, r_2] +N $ where $r_1 =r(G_1, M_1)(d)$, $ r_2=
r(G_2,M_2)(d)$.  Then all points in the intersection of the body $
(\Delta _1(d))_{r(d)-N}$ and of the group  $T_1$ belong to the set
$G_1(d)-dA_1$, all points in the intersection of the body $ (\Delta
_2(d))_{r(d)-N}$ and of the group $T_2$ belong to the set
$G_2(d)-dA_2$. According to the corollary 8.3 each $a$ in the
intersection of the set $(\Delta(d))_{Qr(d)}$ and of the group
$T_1+T_2 $ could be represented as sum of vectors $\bold x$ and
$\bold y$ where $\bold x\in (\Delta_1(d))_{r(d)} $ and $\bold
y\in(\Delta_2(d))_{r(d)} $. According to the lemma 11.1 instead of
the vectors $\bold x$ i $\bold y$ one can use the elements $b$ and
$c$ of the groups $T_1$ and $T_2$ which belong to the sets
$(\Delta _1(d))_{r(d)-N}$ and $ (\Delta _2(d))_{r(d)-N}$.

By definition $r(d)$ is big enough so the points $b$ and $c$ belong
to the sets  $G_1(d)-dA_1$ and $G_1(d)-dA_1$. So each point of the
intersection of the body $(\Delta(d))_{Qr(d)}$ and of the group
$T_1+T_2 $ belongs to the set $G_1(d)+ G_2(d) -d(A_1+A_2)$.

Basically the theorem is proved: to complete the proof we need some
arithmetic calculations.  Using the bodies $C_1(1)$, $C_2(1)$ we
found the constants   $D_1/R_1$ and $D_2/R_2$. Using the groups
$T_1$ and $T_2$ we found the constant $N$. Assume now that the
semigroup  $M_1$ and $M_2$ approximate the semigroups $G_1$ and
$G_2$ and  $r_1=r(G_1,M_1)$, $r_2=r(G_2,M_2)$ are the functions
which appeared in the definition of approximation. The function
$Qr$, where $r =\max [r_1, r_2] +N$ is bigger than $\tilde
r=r(G_1\oplus_t G_2, T_1 +T_2, A_1\oplus_1 A_2)$. Because of the
relations $\lim_{d\rightarrow \infty} r_1(d)/d = \lim_{d\rightarrow
\infty} r_2(d)/d=0$ one can claim that  $\lim_{d\rightarrow
\infty}Qr(d)/d=0$. So $\lim_{d\rightarrow \infty}\tilde r(d)/d=0.$
\end{proof}

\section{Convex body associated to a subspace of regular functions and main theorem}
\subsection{Pre-valuations, valuations  and  Gr\"{o}bner maps} \label{subsec-12}
The original example of a valuation on the space of
meromorphic functions in one variable is the degree of zero or pole of
a function $f$ at a given point say the origin. If $f$ is a Laurent polynomial
this is equal to the degree of the smallest non-zero term of $f$. In higher dimensions
one can take the exponent of the smallest non-zero term of $f$ (with respect to
an ordering of terms) as a valuation.
In this section we will discuss valuations on the field of rational functions
on an irreducible $n$-dimensional (quasi) affine variety $X$. We will use a valuation to
associate a semi-group of integral points to a subspace $L \in K(X)$ of
regular functions. We will then use results of the previous sections on semi-groups
to get our main result on relation between number of solutions of systems of algebraic
equations on $X$ and volume of convex bodies. We will be interested in valuations which
have values in $\Bbb Z^n$ and in particular faithful valuations, i.e. valuations for which
all the integral points in $\Bbb Z^n$ appear as values of the valuation for some rational
function. Classically valuations are used to prove the existence of a unique smooth model for a birational class of algebraic curves. These ideas has been generalized by great classical algebraic geometers e.g. Zariski, to attack the problem of resolution of singularities in higher dimensions
(see \cite{Hodge-Pedoe} for a classical treatment of valuation theory in algebraic geometry).
We should mention that will deal with the so-called {\it non-Archimedean} valuations only.

Let $(I, >)$ be an ordered set. {\it A representation} of $I$ in the
category of vector subspaces of a vector space $V$ is a map, which
associates to each $\alpha \in I$ a non-zero subspace
$V_{\alpha}\subseteq V$ such that if $\alpha<\beta$ then
$V_{\alpha}\subseteq V_{\beta}$ and $\cup_{i\in I} V_i=V$.
Given a non-zero vector  $a\in V$ let
$$I(a) = \{ \alpha \mid a \in V_\alpha\}.$$
Let us say that a representation has the
{\it Gr\"{o}bner property} if for each non-zero vector $a\in V$ the
subset $I(a)\subseteq I$ has a minimum element $v(a)$. The {\it
Gr\"{o}bner map} is the map $v: V \setminus \{0\}\rightarrow I$ which associates to
each non-zero vector the point $v(a)\in I$.
The function $v$ determines the representation, namely
$$V_\alpha = \{a \in V \mid v(a) \leq \alpha\}.$$

Let us say that a representation of an ordered set $I$ with
Gr\"{o}bner property has {\it one-dimensional leaves}, if
whenever $v(a)=v(b)=\alpha$, for $a, b \in V$ then there exist scalars
$\lambda_1, \lambda_2\in \Bbb C$ such that
$v(\lambda_1a+\lambda_1b)> \alpha$.

\begin{Def} A {\it  pre-valuation} on a vector
space $V$ with values in the ordered set $I$ is a representation of
$I$, in the category of vector subspaces of $V$, with
Gr\"{o}bner property and one-dimensional leaves.
\end{Def}

\begin{Ex}[Functions with finite support on an ordered
set] \label{ex-1}
Let $(I, >)$ be an ordered set. Consider the vector space  $V$ of
complex-valued functions on $I$  with finite support.
For each $\alpha \in
I$ let $V_{\alpha}$ be the subspace of $V$ consisting of
functions whose support is contained in the subset $I_{\alpha} =
\{\beta \in I \mid \beta\leq \alpha\}$. The Gr\"{o}bner map $v$
associates to each non-zero function the smallest point in its
support.
\end{Ex}
For a well-ordered ordered set $I$, i.e. a total order such that
any non-empty subset of $I$ has a minimum element,
the condition of finiteness of support in Example \ref{ex-1} can be dropped.
\begin{Ex}
[Functions on a well-ordered ordered set] \label{ex-2}
Let $I$
be a well-ordered set. Consider the vector space  $W$ of
complex valued functions on $I$. For each element $\alpha \in I$
denote by $W_{\alpha}$ the subspace of functions
whose support is contained in $I_{\alpha} = \{ \beta \mid
\beta \leq \alpha\}$. The Gr\"{o}bner map associate to each non-zero
function the smallest point in its support.
\end{Ex}

Let $L\in V$ be a finite dimensional subspace in a vector space $V$ equipped with a
pre-valuation with values in $I$.
\begin{Prop} \label{12.1}
The dimension of $L$ is equal to the number of
points in the image $v(L\setminus \{0\})$ under the Gr\"{o}bner map
$v: L\setminus \{0\}\rightarrow I$.
\end{Prop}
\begin{proof} We prove the claim by induction on the number $k$ of elements in
$v(L)$. Let $a \in L$ be such that $v(a)$ is minimum in $v(L)$.
Using the condition about one-dimensional leaves of the representation
one proves that the pre-image of
$v(L)\setminus \{v(a)\} $ has codimension $1$ in
$L$. By induction hypothesis, the dimension of this pre-image is equal to
$k-1$. The proposition is proved.
\end{proof}

\begin{Ex}[Schubert cells in Grassmannian] \label{ex-3}
Let $\textup{Gr}(n, k)$ be the Grassmannian of $k$-dimensional planes
in $\Bbb C^n$. Let $I=\{1<2\dots<n\}$. In this case the space $V$
from Example \ref{ex-1} naturally identifies with
$\Bbb C^n$. Under the Gr\"{o}bner map each $k$-dimensional
subspace $V\subset \Bbb C^n$ goes to a subset $M\subset I$
containing $k$ elements. The set of all  $k$-dimensional subspaces
which are mapped onto $M$ form the {\it Schubert cell $X_M$} in the
Grassmannian $\textup{Gr}(n, k)$.
$\textup{GL}(n, \Bbb C)$ naturally acts on $Gr(n, k)$ and
the Schubert cells are in fact the orbits of the subgroup of
upper triangular matrices.
\end{Ex}

\begin{Ex}[Schubert cells in complete flag variety] \label{ex-4}
Let $F\ell_n$ be the variety of
all complete flags
$F = (L_0 \subset L_1 \subset \cdots \subset L_n=V)$ where $\dim(L_i) = i$.
Similar to the Grassmannian, one defines cells for $F\ell_n$.
The Schubert cells for $F\ell_n$ are parameterized by permutations $\sigma \in S_n$.
Take a flag of subspaces $F = (L_0 \subset L_1 \subset \cdots \subset L_n=V)$.
Let $M_i = v(L_i)$. By Proposition \ref{12.1}, $\#M_i = i$
and thus $\emptyset \subsetneqq M_0 \subsetneqq \cdots \subsetneqq M_n$.
Let $\{\sigma(i)\} = M_i \setminus M_{i-1}$. Then $\sigma(F) =
(\sigma(1), \ldots, \sigma(n))$ defines a permutation in $S_n$. Given a permutation $\sigma \in S_n$,
$\{ F \in F\ell_n \mid \sigma(F) = \sigma \}$ is the Schubert cell $X_\sigma$.
As in the case of Grassmannian,
the Schubert cells are the orbits of the group of upper triangular matrices
with respect to the natural action of $\textup{GL}(n, \Bbb C)$ on the flag variety.
\end{Ex}

\begin{Def} Suppose the vector space
$V$ has the structure of a commutative algebra over $\Bbb C$ without zero
divisors and the ordered set $I$ has the structure of a
commutative semigroup in which the addition and the ordering are
compatible in a following sense: if $\alpha,\beta\in I$ with $\alpha>\beta$
then for any $\gamma\in I$, $\alpha+\gamma>\beta+\gamma$.
A pre-valuation on the algebra  $V$ with values in $I$ is
called a {\it valuation} if for any two non-zero elements $a,b
\in V$ we have $$v(ab)=v(a)+v(b).$$
\end{Def}

We will mostly be interested in the case when the algebra $V$ is a field
and the valuation take values in the ordered group $\Bbb Z^n$.

\begin{Ex}[Ordering on the semigroup $\Bbb Z^n_+$ and the group $\Bbb Z^n$] \label{ex-5}
Let $\Bbb
Z^n_+$ be the additive semigroup of integral points with non-negative
coordinates in $\Bbb Z^n \subset {\Bbb R}^n$. One defines a well-ordering on $\Bbb Z^n$ as follows:
fix $k\leq n$ independent linear functions $l_1,\dots,l_k$ on $R^n$ such that  functions
$\bold f:\Bbb Z^n_+\rightarrow {\Bbb R}^k$, $\bold l=(l_1,\ldots, l_k)$ is one-to-one. This is always
possible. Let us associate with each point $m \in \Bbb Z^n$ the ordered $k$-tuple of numbers $(l_1(m),\dots,l_k(m))$ and define the ordering on $\Bbb Z^n$ using the lexicographic order in
this set of $k$-tuples of numbers, namely, for $m_1, m_2 \in \Bbb Z_+^n$, we say that $m_1>m_2$
if for some $0 \leq i< k$, $l_1(m_1)=l_1(m_2),\dots, l_i(m_1)=l_i(m_2)$ and
$l_{i+1}(m_1)> l_{i+1}(m_2)$. This gives a total ordering on $\Bbb Z^n$ compatible with addition, induced ordering on $\Bbb Z_+^n$ is a well-ordering. An
ordering on $\Bbb Z^n$ compatible with addition is completely
determined by the induced order on $\Bbb Z^n_+$,
because for any two $m_1,m_2\in \Bbb Z^n$ there is
$m \in \Bbb Z^n$ such that $m_1+m$ and $m_2+m$ lie in $\Bbb Z^n_+$
\end{Ex}

A point $m=(m_1,\dots,m_n) \in \Bbb Z^n_+$ can be identified with
the  monomial  $x^m=(x_1^{m_1}\dots x_n^{m_n})$. Under this
identification the addition in $\Bbb Z^n$ corresponds to the
multiplication of monomials.

\begin{Ex}[Gr\"{o}bner ordering on the algebra of polynomials] \label{ex-6}
Under the identification of integral points and monomials, a function $c:\Bbb Z^n_+\rightarrow
\Bbb C$ with finite support corresponds
to a polynomial $P(x_1,\dots,x_n)= \sum c(m)x^m$. Consider the well-ordered set $I=\Bbb Z^n$ with
the ordering in Example \ref{ex-5}.
Using the pre-valuation in Example \ref{ex-1} on the set of functions with finite
support we get a pre-valuation on the algebra of polynomials with a Gr\"{o}bner mapping $v$ from the set of
non-zero polynomials to $\Bbb Z^n_+$. One verifies that this is compatible with multiplication of
polynomials and is in fact a valuation.
\end{Ex}

This example can be naturally generalized to the algebra of power series.

\begin{Ex}[Gr\"{o}bner ordering on the algebra of formal power
series and the algebra of germs of analytic functions] \label{ex-7}
To a function $c:\Bbb Z^n_+\rightarrow \Bbb C$ one associates a
formal power series  $P(x_1,\dots,x_n)= \sum_{m \in \Bbb Z_+^n} c(m)x^m$. The
construction Example \ref{ex-2} in this case gives a pre-valuation, together with a Gr\"{o}bner ordering,
on the algebra of formal power series. Again this is compatible with multiplication and is
in fact a valuation.
\end{Ex}

\begin{Ex}[Gr\"{o}bner valuation on the field of rational
functions] \label{ex-8}
Valuation in Example \ref{ex-6} can be extended to
a valuation on the field of rational functions with values in
the ordered group $\Bbb Z^n$. In fact each rational function  $R$ is a
quotient of two polynomials $R=P/Q$. For $R \neq 0$ define $v(R)=v(Q)-v(P)$. This
is well-defined, i.e. is independent of the choice of $P$ and $Q$,
and gives a valuation on the field the rational functions.
\end{Ex}

In the same way as in the previous example one defines a valuation on
the quotient field of algebra of formal power series and on the field of germs of meromorphic
functions.

We will use a valuation on the field of rational functions on an
irreducible (quasi) affine algebraic variety $X$ which take values in the
group $\Bbb Z^n$. We say that a valuation is {\it faithful} if
it is onto, i.e. takes all the values in $\Bbb Z^n$.

\begin{Ex}[Gr\"{o}bner valuation on the field of rational
functions on an affine variety] \label{ex-9}
Let $X$ be an irreducible
$n$-dimensional (quasi) affine algebraic variety and let $f_1,\dots, f_n$ be
regular functions on $X$. Assume that a some smooth
point $a\in X$ is a common zero of all the $f_i$ and their
differentials $df_i$ at $a$ are independent. Then in a neighborhood of the
point $a$ the functions $f_1,\dots, f_n$ define a local coordinate
system on $X$. Fix a well-ordering ordering  in the semigroup of
monomials in the $f_i$ which is isomorphic to the semigroup $\Bbb Z^n_+$.
As in Examples 6,7, this Gr\"{o}bner valuation can be be extended to
a Gr\"{o}bner valuation on the algebra of germs of analytic functions about the point  $a$ and to a
Gr\"{o}bner valuation on the field of meromorphic
functions. In particular we obtain a faithful
valuation on the field of rational functions  on $X$ with
values in $\Bbb Z^n$.
\end{Ex}

One can modify valuations from the Example \ref{ex-9} to associate a similar
valuation to a singular point on $X$.

\begin{Ex}[Gr\"{o}bner valuation on field of rational
functions constructed from a Parshin point on $X$] \label{ex-10}
Consider a sequence of
maps  $$\{a\}=X_0 \stackrel{\pi_0}{\to}  X_1 \stackrel{\pi_1}{\to} \cdots
\stackrel{\pi_{n-1}}{\to} X_n = X,$$ where
each $X_i$, $i=0,\dots n-1$, is a normal irreducible variety of dimension $i$
and the map $X_i \stackrel{\pi_i}{\rightarrow} X_{i+1}$ is a normalization map for the image
$\pi_i(X_i)\subset X_{i+1}$. Such a sequence represents a {\it
Parshin point} on the variety $X$. A collection of rational functions
$f_1,\dots,f_n$ represents {\it a  system of  parameters} about such
the $X_i$, if for each $i$, the function $\pi^*_i\circ \dots\circ \pi^*_nf_k$ on the
hypersurface  $\pi_{i-1}(X_{i-1})$ in the normal variety $X_i$ has a
zero of first order. Given a sequence of the $X_i$ and a system of parameters,
one can associate a {\it iterated Laurent series} to any rational function $F$.
Iterated   Laurent series can be defined inductively. It is a usual
Laurent series  $\sum_{k} c_kf_n^k$
with a finite number of terms with negative degrees in the variable $f_n$ and each coefficient $c_k$ in
which is an iterated  Laurent series in the variables
$f_1,\dots,f_{n-1}$. Each  iterated Laurent series has a monomial
$f_1^{k_1}\dots f_n^{k_n}$ of the smallest  degree  with respect to
the  lexicographic order in degrees $(k_1,\dots,k_n)$ (where first we
compare the degrees $k_n$, then the degrees $k_{n-1}$ and so on).
The map which assigns to a Laurent series its smallest monomial
defines a faithful valuation on the field on rational functions on
$X$.
\end{Ex}

\subsection{Hilbert Theorem} \label{subsec-13}
Let $X$ be an $n$-dimensional irreducible (quasi) affine algebraic variety
and let $L \in K(X)$ be a finite dimensional vector space of regular
functions on $X$. As in Section \ref{subsec-1}, the subspace $L$ gives rise to a
map $$\Phi_L: X \to \Bbb P(L^*),$$ where $L^*$ is the vector space dual of $L$.
Let $Y = \Phi_L(X)$. The following is a version of the
classical theorem of Hilbert. It plays a key role for us.

\begin{Th}[Hilbert's theorem] Let $H$ be the Hilbert function of
$(X, L)$, defined by $H(k)=\dim L^k$. Then for large values of $k$,
the function $H$ becomes a polynomial in $k$. Moreover, the degree $m$ of this Hilbert
polynomial is equal to the dimension of the variety $Y$, and the
leading coefficient $c$ in the Hilbert polynomial is the degree of $Y\subset \Bbb P(L^*)$
divided by $m!$.
\end{Th}

\begin{Cor} \label{13.1}
For dimension $H(k)$ of the space $L^k$, there are numbers $0 \leq m \leq n$
and $c > 0$ such that
$$\lim_{k\rightarrow \infty}\frac{H(k)}{k^m}=c.$$
2) $m$ and $c$ have the following
properties: consider a system of equations $f_1=\dots=f_n=0$ on the
variety $X$, where $f_1,\dots,f_n$ are a general $n$-tuple of
functions in $L$. If $m<n$ then the  system has
no roots on $X$. If $m=n$ then the system has $n!c d$ roots on $X$,
where $d$ is a mapping  degree for $\Phi_L:X \to Y \subset \Bbb P(L^*)$.
\end{Cor}

\subsection{The graded semigroup and the Newton convex body of a
subspace of regular functions: Main theorem} \label{subsec-14}
Fix a faithful valuation $v$ on the field of rational
functions on a (quasi) affine irreducible variety $X$ with values
in the group  $\Bbb Z^n$ (see Examples 9-10 in section 12). Using
this valuation we associate a graded semi-group to each finite dimensional space of
regular functions $L\subset K(X)$.

\begin{Def} The {\it Gr\"{o}bner semi-group} $G(L)$
of the space $L\in K(X)$ is the following semi-group.
$$G(L) = \bigcup_{p} \{(p, m) \mid m \in v(L^p \setminus \{0\})\} \subset \Bbb Z \times \Bbb Z^n.$$
\end{Def}

\begin{Prop} \label{14.1}
1) For each space $L\in K(X)$ the
Gr\"{o}bner semi-group $G(L)$ is a graded
semigroup in $\Bbb Z \times \Bbb Z^n$. The semigroup $G(L)$ has finite sections and a limited
growth (see Section \ref{subsec-10}). 2) For any two spaces
$L_1,L_2\in K(X)$, the semigroup $G(L_1L_2)$ contains the semigroup
$G(L_1)\oplus_t G(L_2)$ (see Section \ref{subsec-11}).
\end{Prop}
\begin{proof} 1) The number of points in the set $G(L)$ with the first
coordinate equal to  $p$ is the dimension of the space $L^p$
(Proposition \ref{12.1}). So the number of points in any section of $G(L)$ is
finite and not equal to zero for every  $p \in \Bbb N$. By
definition points $(p_1,m_1)$, $(p_2,m_2)$ belong to $G(L)$, if for
some functions  $f_1\in L^{p_1}$ and $f_2\in L^{p_2}$ we have
$v(f_1)=m_1$, $v(f_2)=m_2$. The function $f_1f_2$ belongs to
the space $L^{p_1}L^{p_2}$ and $v(f_1f_2)=m_1+m_2$. So the point
$(p_1+p_2,m_1+m_2)$ belongs to the set $G(L)$. Thus
$G(L)$ is a graded semigroup with finite sections. According to
Corollary \ref{13.1} the semigroup  $G(L)$ has limited growth.

2) By definition for each point $(p,m)$ in the semigroup $G(L_1)\oplus_t G(L_2)$
there are points $(p,m_1)$ and $(p,m_2)$ in the semigroups $G(L_1)$
and $G(L_2)$ such that $m=m_1+m_2$. And by definition of the semigroups
$G(L_1)$ and $G(L_2)$ there are function $f_1\in L^p_1$ and $f_2\in
L^p_2$ such that $v(f_1)=m_1$ i $v(f_2)=m_2$. The function $f_1f_2$
belongs to $(L_1L_2)^p$ and $v(f_1f_2)=m_1+m_2$.
Hence $G(L_1)\oplus_t G(L_2)$ is contained in $G(L_1L_2)$.
\end{proof}

\begin{Lem} \label{14.2}
Given a faithful valuation $v$ on the field of rational functions on $X$ with values
in $\Bbb Z^n$ and a finite set $P\subset \Bbb Z^n$
with $k$ elements, there exists a $k$-dimensional space $L \in K(X)$
of regular functions such that its image under $v$
equals to $P+m$, for some $m \in \Bbb Z^n$.
\end{Lem}
\begin{proof} Since the valuation is faithful,
there is a finite dimensional space $\bar
L$ of  rational function such that its image under $v(\bar{L} \setminus \{0\}) = P$.
We know the dimension of $\bar L$ equals $k$, the number
of points in $P$. On can find a rational function $g$ such that, after multiplication by $g$,
all the functions in
$\bar L$ by become regular functions. The image of
the space $L=g\bar L$ is equal to $P+m$ where $m=v(g)$.
\end{proof}

\begin{Prop} \label{14.3}
For each space $L\in K(X)$ the
semigroup $G(L)$ is contained in some  graded semigroup of complete
rank and with limited growth.
\end{Prop}
\begin{proof} Clearly if we enlarge a subspace $L$ then its semi-group $G(L)$
becomes bigger (or remains the same).
From Proposition \ref{14.2} it follows that
we can enlarge the space $L$ (inside $K(X)$) so that the
semigroup $G(L)$ becomes of complete rank. According to Proposition \ref{14.1}
all semigroups $G(L)$ for $L\in K(X)$ have limited
growth.
\end{proof}

Let us summarize.
Let $X$ be an irreducible (quasi) affine variety of dimension $n$.
Fix a faithful valuation $v$ on the field of rational functions $\Bbb C(X)$
with values in $\Bbb Z^n$. The valuation $v$ associate to a space
$L\in K(X)$ the graded semigroup $G(L)$, which is contained in a
graded semigroup of complete rank, and limited growth. To each such
semigroup corresponds its Newton convex body $\Delta(G(L))$ and the
index $ind(G(L))$ (Section \ref{subsec-10}). We are now ready to state our main theorem.

\begin{Th}[Main theorem] \label{14.4}
If the Newton convex body $\Delta(G(L)$ of the
space $L\in K(X)$ has dimension $n$, then the intersection
index $[L,\dots,L]$ of $n$ copies of the space $L$ is equal to
$$n! V_n (\Delta(G(L))p(L)/ind (G(L)),$$ where $p(L)$  is the
mapping degree $\Phi_L: X\rightarrow \Bbb P(L^*)$.
If the  Newton convex body $\Delta(G(L))$ has
dimension smaller than $n$, then the image of  $X$ under the map
$\Phi_L$ has dimension smaller than $n$ and we have $$[L,\dots,L]=0.$$
The Newton convex body $\Delta (G(L_1L_2))$ of the product  of two spaces
$L_1,L_2\in K(X)$ is contained in the Minkowski sum  $\Delta
(G(L_1))+ \Delta (G(L_2))$ of the Newton domains of those spaces. If
the Newton domains  $\Delta (G(L_1))$, $\Delta (G(L_2))$ have
dimensions $n$, then $ind(G(L_1L_2))$ is not greater then each of
the indices $ind(G(L_1))$ and $ind(G(L_2))$.
\end{Th}
\begin{proof} The main theorem is already proved. We reduced it to
the description of the graded semigroups and their sums and to the
Hilbert theorem.
\end{proof}

\subsection{Convex body associated to a line bundle over a projective variety}
Now let $Y$ be an irreducible projective variety of dimension $n$ and $\L$ an
ample line bundle on $Y$. The ring of sections of $\L$ is defined as
$$R = \bigoplus_{k=0}^\infty H^0(Y, \L^{\otimes k}).$$
Take a valuation $v: R \to \Bbb Z^n$ on the ring of sections.
In fact, the constructions from Section \ref{subsec-12}, in particular, Example \ref{ex-10},
applies in the same way to give valuation on the ring of sections.

Assign a semi-group $G(\L)$ and a convex set
$\Delta(G(\L))$ to $\L$ in the same way as in Section \ref{subsec-14}, replacing
the subspace $L$ with $H^0(Y, \L)$ and $L^k$ with $H^0(Y, \L^{k})$.

\begin{Def}
The number $c_1(\L)^n$, where
$c_1(\L)$ is the first Chern class, is called {\it degree} of $\L$.
It is equal to the number of solutions of a generic system
$\sigma_1(y) = \cdots = \sigma_n(y) = 0$ of holomorphic sections of $\L$.
\end{Def}
As in Section \ref{subsec-14}, define the map
$$\Phi_{\L}: Y \to \Bbb P(H^0(X, \L)).$$
Since $\L$ is ample, the subvariety $\Phi_{\L}(Y)$ has dimension $n$.
If $\L$ is very ample then $\Phi_{\L}$ is an embedding. Let $p(\L)$ be the degree of the map $\Phi_{\L}$.
Let $ind(G(\L))$ be the index of $G(\L)$ in $\Bbb Z \times \Bbb Z^n$ (which is finite by the above assumption). The following is the analogue of Theorem \ref{14.4} for projective varieties.
\begin{Th}\label{14.4-projective}
Let $Y$ be a projective variety of dimension $n$ and $\L$ an
ample line bundle on $Y$. Let $v$ be a valuation on $R$,
the ring of holomorphic sections of $\L$. Then
$$\deg(\L) = n! V_n (\Delta(G(\L))p(\L)/ind(G(\L)).$$
\end{Th}
The proof is the same as the proof of Theorem \ref{14.4}, but instead of
Hilbert theorem we use its projective version, namely asymptotic Riemann-Roch theorem.

\begin{Ex}[Gelfand-Cetlin and string polytopes] \label{ex-GC}
In representation theory, to any irreducible representation $V_\lambda$ of
$\textup{GL}(n, \Bbb C)$, with highest weight $\lambda$,
there corresponds a so-called {\it Gelfand-Cetlin polytope} $\Delta_\lambda$.
The integral points in this polytope parameterize the elements of a natural basis
for $V_\lambda$. Using the heavy algebraic machinery of crystal graphs and canonical bases
these construction has been generalized to any complex connected reductive algebraic
group $G$. The resulting polytopes are called {\it string polytopes} (see for example
\cite{Littelmann}). Similarly the integral points in a string polytope,
associated to a dominant weight $\lambda$, parameterize
the elements of a natural basis for $V_\lambda$. It is shown in \cite{Kaveh} that a string polytope
corresponding to $\lambda$, in fact, coincides with the Newton polytope $\Delta(G(\L_\lambda))$ for the flag variety $Y=G/B$ and the $G$-line bundle $\L_\lambda$. The valuation $v$ is a valuation coming from a Parshin point corresponding to a sequence of Schubert varieties. The special case of $G = \textup{SP}(2n, \Bbb C)$ has been proved by Okounkov earlier \cite{Okounkov-Newton-polytope}.
\end{Ex}

One can identify the space of holomorphic sections $H^0(Y, \L)$
with a subspace of $\Bbb C(Y)$, and hence with a subspace of regular functions on an open
(quasi) affine subvariety $X$ of $Y$:
Take $s_0$ to be a holomorphic section of the
line bundle $\L$ with $D = \textup{Div}(s_0)$. Then any holomorphic
section $s$ is equal to $fs_0$ with $f \in \Bbb C(Y)$ satisfying
$(f) + D \geq 0$. Thus $H^0(Y, \L)$ identifies
with the subspace $$L = L(D) = \{ f \in \Bbb C(Y) \mid (f) + D \geq 0\}.$$
Similarly, for any $k$, $H^0(Y, \L^k)$ identifies with $L(kD)$.
Let $X$ be an open (quasi) affine subvariety of $Y \setminus D$. Then
all the functions in the $L(kD)$ are regular on $X$. Thus they can be regarded as
subspaces of $\mathcal{O}(X)$, in fact, $L(kD) \in K(X)$.
In general $L^k$ is only a subset of $L(kD)$ and  $L(kD)$ could be bigger than $L^k$.
Thus the case of sections of a line bundle on a projective variety $Y$, while similar,
is slightly different than the case of subspaces of regular functions on an (quasi) affine variety $X$.

Take a Parshin point
$$\{a\}=X_0 \stackrel{\pi_0}{\to}  X_1 \stackrel{\pi_1}{\to} \cdots
\stackrel{\pi_{n-1}}{\to} X_n = X,$$
in $X$. As in Example \ref{ex-10} this Parshin point gives a
valuation on $\Bbb C(X)$ and the ring of sections $R$ of any line bundle $\L$.
Since the divisor of $s_0$ is supported outside
$X$ then $v(s_0) = 0$ and hence $v(s) = v(fs_0) = v(f)$. That is, under the identification
$s \mapsto f$, of sections of $\L$ with regular functions on $X$, the valuations on $R$ and on $\Bbb C(X)$
agree.

Conversely, a (quasi) affine variety $X$ and subspace $L \in K(X)$ gives rise
to a projective variety $Y$ and a line bundle $\L$ as follows:
consider $\Phi_L:X \to \Bbb P(L^*)$. Put $Y = \overline{\Phi_L(X)}$ and let
$\L$ be the line bundle on $Y$
induced by the canonical line bundle $\mathcal{O}_{\Bbb P(L^*)}(1)$.
As before fix a full valuation $v$ on $\Bbb C(X)=\Bbb C(Y)$.
Let $L$ be very ample. The following theorem shows that the Newton convex body for $(X, L)$
is the same as that of $(Y, \L)$.
\begin{Th} \label{th-DeltaX-DeltaY}
$\Delta(G(\L)) = \Delta(G(L))$ and $ind(G(\L)) = ind(G(L))$.
\end{Th}
\begin{proof}
For every $k$, $L^k$ can be identified with a subspace of $H^0(Y, \L^{\otimes k})$. Thus
$G(L) \subseteq G(\L)$ implying that $ind(G(L)) \geq ind(G(\L))$ and
$\Delta(G(L)) \subseteq \Delta(G(\L))$. That is,
$$\frac{1}{ind(G(L))}V_n(\Delta(G(L)) \leq \frac{1}{ind(G(\L))}V_n(\Delta(G(\L)).$$
But from definition we know $\deg(L) = \deg(\L)$
which by Theorems \ref{14.4} and \ref{14.4-projective} implies that
$$ n! V_n(\Delta(G(L))/ind(G(L)) = n! V_n(\Delta(G(\L))/ind(G(\L))$$
This shows that $ind(G(L)) = ind(G(\L))$ and $\Delta(G(L)) = \Delta(G(\L))$.
\end{proof}

\subsection{Case of a Hamiltonian group action and relation with the moment polytope}
Let $T$ be the algebraic torus $(\Bbb C^*)^k$ and $Y$ a (smooth) projective
$T$-variety equipped with a $T$-equivariant (very ample) line bundle
$\L$. The variety $Y$ gets a symplectic structure from the projective
embedding associated with $\L$. Equipped with this symplectic
structure, $Y$ becomes a Hamiltonian $T_{\Bbb R}$-space where $T_{\Bbb R}$ is
the real torus $(S^1)^k$. Let
$$\mu: Y \to \textup{Lie}(T_{\Bbb R})^*$$ be the moment map and $\mu(Y, \L)$ be the
moment polytope. Choose a $T$-stable Parshin point (see Section \ref{ex-10}) namely a sequence
$$\{a\}=X_0 \stackrel{\pi_0}{\to}  X_1 \stackrel{\pi_1}{\to} \cdots
\stackrel{\pi_{n-1}}{\to} X_n = Y,$$ where each $X_i$, $i=0,\dots n-1$,
is a normal irreducible $T$-variety of dimension $i$
and the map $X_i \stackrel{\pi_i}{\rightarrow} X_{i+1}$ is a normalization map for the image
$\pi_i(X_i)\subset X_{i+1}$.
Such a Parshin point always exists.
Let $v$ be the associated valuation.
Let $V$ and $W$ denote the real span
of $\Delta(G(\L))$ and $\mu(Y, \L)$ respectively. Put
$\dim V = m$ and $\dim W = r$.

\begin{Th} \label{thm-pi} \cite{Okounkov-Brunn-Minkowski}
There is a natural linear map $\pi: V \to W$ with
$$\pi(\Delta(G(\L)) = \mu(Y, \L).$$
Moreover, for every $\lambda \in \mu(Y, \L)$, the $(m-r)$-dimensional volume of
the fibre $\pi^{-1}(\lambda)$ is equal to the Duistermaat-Heckman
piecewise polynomial measure $p(\lambda)$.
\end{Th}

\begin{Rem}
If the convex body $\Delta(G(\L))$ is a polytope then
the fact that Duistermaat-Heckman measure is equal to the $(m-r)$-dimensional
volume of $\pi^{-1}(\lambda)$ implies its piecewise polynomiality.
\end{Rem}
In fact in \cite{Okounkov-Brunn-Minkowski} the author considers the more general case of a reductive group $G$ acting on $Y$. In this case he applies the construction of the convex set $\Delta$ to the
smaller subalgebra of $U$-invariant functions ($U$ being the maximal unipotent subgroup of $G$).
As a result he obtains a convex set which is smaller than the convex set $\Delta(G(\L))$ we
considered. In particular, Okounkov's convex set in general could have dimension smaller than $\dim(Y)$.

\section{Applications}
\subsection{Theorems of Kushnirenko and Bernstein}
\label{subsec-15}
The well-known theorems of Kushnirenko and Bernstein are particular
cases of Theorem \ref{14.4}. In this section we will discuss these
theorems and show how they follow from Theorem \ref{14.4}.
The proof below more or less the same as the proof of these results in \cite{Khovanskii-finite-sets}.
In fact the present paper should be considered as an unexpected and far-reaching
generalization of \cite{Khovanskii-finite-sets}). We will just sketch the proofs.

Let $X$ be the affine variety $(\Bbb C^*)^n$ with the
coordinates  $x_1,\dots, x_n$. In the semigroup $K((\Bbb C^*)^n)$ of
finite dimensional spaces of regular functions on $(\Bbb C^*)^n$
there is a chosen subsemi-group $K_i((\Bbb C^*)^n)$, which contains
all finite dimensional spaces, invariant under the group action.
Theorems of Kushnirenko and Bernstein describe the intersection
index in the semigroup $K((\Bbb C^*)^n)$ in  geometrical terms. Let
us start with definitions.

Each point $m=(m_1,\dots, m_n)$ in the group $\Bbb Z^n$ corresponds
to the monomial $x^m=x_1^{m_1}\dots x_n^{m_n}$ (note that the
monomials are the characters of the group $(\Bbb C^*)^n$). Each
regular function $f$ on the  group $(\Bbb C^*)^n$ is a {\it Laurent
polynomial}, i.e. is a  linear  combination of monomials $f=\sum
c_mx^m$. To each regular function $f$ one can associate its support
$supp (f)$ --- the finite set $M$ in the group  $(\Bbb C^*)^n$ which
consists of all points $m$, such that the monomial $x^m$ appears in
the representation of the function $f$ in the form of  Laurent
polynomial $f=\sum_{m\in M} c_mx^m$,  with a non zero coefficient
$c_m\neq 0$. The Newton polyhedron  $\Delta(f)\subset {\Bbb R}^n$ of
$f$ is the convex hall of the support  $supp (f)\subset \Bbb
Z^n\subset {\Bbb R}^n$ of the function $f$. For each finite set
$M\subset \Bbb Z^n$ denote by $L(M)$ the vector space of Laurent
polynomials $f$, which support belongs to the set $M$, $sup
(f)\subset M$. Fix any  Gr\"{o}bner ordering on the lattice $\Bbb Z^n$.
One can easily proof the following statement

\begin{Prop} \label{15.1}
Each space in the semigroup $K_i (\Bbb
(C^*)^n)$  is a space $L(M)$ for some finite subset $M$ in the
lattice $\Bbb Z^n$. The Gr\"{o}bner map maps the  space  $L(M)$ into the
set $M$ (in particular, the image of the space $L$ under the Gr\"{o}bner
map is independent on a Gr\"{o}bner ordering ). For each couple of
spaces  $L_1,L_2\in K_i (\Bbb (C^*)^n)$ the following relation holds
$Gr (L_1L_2)=v(L_1)+v(L_2)$. The Newton convex body of the space $L\in
K_i (\Bbb (C^*)^n)$ coincides with the convex hall of the set $M$.
\end{Prop}

Assume that the convex hall of the set $M$ has the dimension $n$. It
is easy to   see that the degree of the natural map $\Bbb
(C^*)^n\rightarrow PL(M)^*$ is equal to the index of the subgroup in
$\Bbb Z^N$ generated by the set $M$.

Using these facts one can reduce from the main theorem the following
results
\begin{Th}[Kushnirenko] Let $f_1,\dots, f_n$ be a generic
$n$ -tuple of Laurent polynomials with fixed Newton polyhedra
$\Delta$. Then the number of the roots on $(\Bbb C^*)^n$ of the
system  $f_1=\dots=f_n=0$ is equal to  $n!V(\Delta)$.
\end{Th}

\begin{Th}[Bernstein] Let $f_1,\dots, f_n$ be a generic
Laurent polynomials  with the Newton polyhedra $\Delta_1,\dots,
\Delta_n$. Then the number of the roots on $(\Bbb C^*)^n$ of the
system  $f_1=\dots=f_n=0$ is equal to multiplied by $n!$ Minkowski
mixed volume of this polyhedra, e.i. is equal to
$n!V(\Delta_1,\dots, \Delta _n)$.
\end{Th}

\subsection{Brunn--Minkowski inequality, its corollaries and
generalizations} \label{subsec-16}
Let $\Delta_1,\Delta_2\subset {\Bbb R}^n$  be bounded convex bodies
and let $\Delta= \Delta_1+\Delta_2$ be their Minkowski sum. The
following important and simple inequality was discovered by Brunn.

\begin{Th}[Brunn--Minkowski inequality]
$$V_n^{\frac{1}{n}}(\Delta_1)+V_n^{\frac{1}{n}}(\Delta_2)\leq
V_n^{\frac{1}{n}}(\Delta).$$
\end{Th}
When $n=2$, i.e. on the plane, the Brunn--Minkowski inequality has the
following form.

\begin{Th}[Isoperimetric inequality for planar regions]
Areas $V_2$ of the
bounded convex planar bodies  $\Delta_1$ and $\Delta_2$ and their
mixed area  $V_2(\Delta_1, \Delta_1)$ are related by the following
inequality $ V_2(\Delta_1)V_2(\Delta_2)\leq
V_2^2(\Delta_1,\Delta_2).$
\end{Th}
\begin{proof} From the Brunn--Minkowski inequality we have the
following relation
$(V_2^{\frac{1}{2}}(\Delta_1)+V_2^{\frac{1}{2}}(\Delta_2))^2\leq
V_2(\Delta_1+\Delta_2)=
V_2(\Delta_1)+2V_2(\Delta_1,\Delta_2)+V_2(\Delta_2),$ which is
equivalent to  $ V_2(\Delta_1)V_2(\Delta_2)\leq
V_2^2(\Delta_1,\Delta_2).$
\end{proof}

Let $l(\partial\Delta)$ denote the perimeter of a two dimensional
convex body $\Delta$. Let $B_1$ be the unit ball centered at the origin. It is
easy to see that  $V_2(\Delta,B_1)=\frac{1}{2}l(\partial\Delta)$. If
$\Delta_2=B_1$ the isoperimetric inequality becomes
$$V_2(\Delta)V_2(B_1)\leq \frac{1}{4}l^2(\partial\Delta),$$ and hence
$$V_2(\Delta)\leq \frac{1}{2}l^2(\partial\Delta)/\pi.$$ Thus we obtain
an estimate of the area  $V_2(\Delta)$ of the body $\Delta)$ in terms
of its perimeter. This classical estimate known as {\it isoperimetric inequality}
is sharp. That is why the two dimensional case of the Brunn -- Minkowski
inequality is also called isoperimetric inequality.

The Brunn--Minkowski inequality has numerous generalizations. Some
of them (e.g. Alexandrov--Fenchel inequality), have rather
complicated proofs. Here we list some generalizations and corollaries of the
Brunn-Minkowski inequality.

\begin{Th}[Alexandrov--Fenchel inequality] Let
$\Delta_1,\dots,\Delta_n\subset {\Bbb R}^n$ be bounded convex bodies
in ${\Bbb R}^n$. Denote by $V_n(\Delta)$ the volume of  $\Delta$ and
by $V_n(\Delta_1,\dots,\Delta_n)$  the mixed volume of
$\Delta_1,\dots,\Delta_n$. The following inequality holds
$$V_n(\Delta_1,\Delta_2, \Delta_3\dots,\Delta_n)^2\geq V_n(\Delta_1,\Delta_1,
\Delta_3\dots,\Delta_n)V_n(\Delta_2,\Delta_2,\Delta_3\dots,\Delta_n).$$
\end{Th}

The following inequalities are formal corollaries from the
Alexandrov--Fenchel inequality.
\begin{Cor}[Corollaries of Alexandrov--Fenchel inequality] Let
$P$, $Q$ and $\Delta_i$ be bounded convex bodies in ${\Bbb R}^n$. The
following inequalities hold:

\noindent{(a)} $~~V(\Delta_1, \ldots, \Delta_n)^m \geq \prod_{i=1}^m
V(\underbrace{\Delta_i, \ldots, \Delta_i}_{m}, \Delta_{m+1}, \ldots, \Delta_n).$\\

\noindent{(b)} $~~V(\Delta_1, \ldots, \Delta_n) \geq \Vol(\Delta_1) \cdots \Vol(\Delta_n).$\\

\noindent{(c)} $~~V(\underbrace{P, \ldots, P}_{i},
\underbrace{Q, \ldots, Q}_{m-i}, \Delta_{m+1}, \ldots, \Delta_n)
\geq V(\underbrace{P \ldots, P}_{m}, \Delta_{m+1}, \ldots, \Delta_n)^i
\cdot V(\underbrace{Q, \ldots, Q}_{m},
\Delta_{m+1}, \ldots, \Delta_n)^{m-i}.$\\

\noindent{(d)} $~~V(\underbrace{P, \ldots, P}_{k}, \underbrace{Q, \ldots, Q}_{l},
\Delta_{k+l+1}, \ldots, \Delta_n)^2
\geq \\ V(\underbrace{P, \ldots, P}_{k-1}, \underbrace{Q, \ldots, Q}_{l+1}, \Delta_{k+l+1}, \ldots, \Delta_n)
\cdot V(\underbrace{P, \ldots, P}_{k+1}, \underbrace{Q, \ldots, Q}_{l-1},
\Delta_{k+l+1}, \ldots, \Delta_n).$\\
\end{Cor}

\subsection{Algebraic analogue of Brunn--Minkowski and
Alexandrov--Fenchel inequalities and their corollaries} \label{subsec-17}
Let $L_1,L_2\in K(X)$ be finite dimensional spaces of regular
functions on an irreducible (quasi) affine variety $X$. Consider the
self-intersection indices $[L_1,\dots, L_1]$, $[L_2,\dots, L_2]$ and
$[L_1L_2,\dots, L_1L_2]$ of the spaces $L_1$, $L_2$ and $L_1L_2$.
Assume that the spaces $L_1$ and $L_2$ separate generic points on $X$,
i.e. the mapping degrees of the maps $\Phi_{L_1}$ and $\Phi_{L_2}$ are $1$.

\begin{Th}[Analogue of Brunn--Minkowski inequality for
self-intersection index] \label{17.1}
If the spaces $L_1$ and $L_2$ separate
generic point on an irreducible (quasi) affine algebraic variety $X$, then we have
$$[L_1,\dots,L_1]^{\frac{1}{n}}+[L_2,\dots,L_2]^{\frac{1}{n}}\leq
[L_1L_2,\dots,L_1L_2]^{\frac{1}{n}}.$$
\end{Th}
\begin{proof} The inequality follows from Theorem \ref{14.4} and
the classical Brunn--Minkowski inequality:  $L_1$ and $L_2$ separate
generic points on $X$ and hence dimensions of their Newton convex bodies
$\Delta (G(L_1))$ and $\Delta (G(L_2))$ are equal to $n$. Newton convex body
$\Delta (G(L_1L_2))$ of the space $L_1L_2$ contains
$\Delta (G(L_1))+ \Delta (G(L_2))$. Also index  of the group generated by
the semigroup $G(L_1L_2)$ is not bigger than minimum of indices of
the groups generated by the semigroups $G(L_1)$ and $G(L_2)$. Applying
the Brunn--Minkowski inequality to the
Newton convex bodies $\Delta (G(L_1))$, $\Delta (G(L_2))$ and $\Delta
(G(L_1L_2))$ we get the required inequality.
\end{proof}

When $X$ is a surface (i.e. $n=2$) from
the algebraic analogue of Brunn--Minkowski one obtain the
following.

\begin{Th}[Affine version of Hodge Index Theorem] \label{17.2}
$$ [L_1,L_1] [L_2,L_2]\leq [L_1,L_2]^2.$$
\end{Th}
\begin{proof} From Theorem \ref{17.1} we have
$([L_1,L_1]^{\frac{1}{2}}+
 [L_2,L_2]^{\frac{1}{2}})^2\leq [L_1L_2, L_1L_2]=
[L_1,L_1]+2[L_1,L_2]+[L_2,L_2],$ and this is equivalent to the inequality
$ [L_1,L_1][L_2,L_2]\leq [L_1,L_2]^2.$
\end{proof}

In Section \ref{subsec-7} using the affine version of Hodge Index Theorem we proved
the following.

\begin{Th}[Algebraic analogue of Alexandrov--Fenchel
inequality] \label{17.3}
Let $X$ be an irreducible $n$-dimensional (quasi) affine
variety and $L_1,\dots,L_n\in K(X)$ very ample spaces. Then
$$[L_1,L_2, L_3\dots,L_n]^2\geq [L_1,L_1,
L_3\dots,L_n][L_2,L_2,L_3\dots,L_n].$$
\end{Th}

The following are formal corollaries from the algebraic
analogue of Alexandrov--Fenchel inequality.

\begin{Cor}[Corollaries of the algebraic analogue of
Alexandrov--Fenchel inequality]
Let $X$ be an $n$-dimensional (quasi) affine irreducible
variety and let $P$, $Q$ and $L_i$ be very ample
spaces from the semigroup $K(X)$. The following hold:

\noindent{(a)} $[L_1,\dots,L_n]^m \geq\prod_{i=1}^m [\underbrace{L_i,\dots,
L_i}_m, L_{m+1}\dots, L_n];$

\noindent{(b)} $[L_1,\dots,L_n] \geq [L_1,\dots,L_1]\dots [L_n,\dots,L_n];$

\noindent{(c)} $[\underbrace{P,\dots,P}_i,\underbrace{Q,\dots,
Q}_{m-i},L_{m+1},\dots,L_n] \geq
[\underbrace{P,\dots,P}_m,L_{m+1},\dots,L_n]^i\cdot$

$[\underbrace{Q,\dots, Q}_m, L_{m+1},\dots,L_n]^{m-i};$

\noindent{(d)} $[\underbrace{P,\dots,P}_k,\underbrace{Q,\dots,
Q}_{l+1},L_{k+l+1},\dots,L_n]^2 \geq$

$[\underbrace{P,\dots,P}_{k-1},\underbrace{Q,\dots,
Q}_{l+1},L_{k+l+1},\dots,L_n][\underbrace{P,\dots,P}_{k+1},\underbrace{
Q,\dots, Q}_{l-1},L_{k+l+1},\dots,L_n].$
\end{Cor}

Inequalities (a)-(d) follow from the algebraic analogue of the
Alexandrov--Fenchel inequality exactly in the same way as the
similar geometrical inequalities follows from the
Alexandrov--Fenchel inequality.

\begin{Rem}
Of course all the previous inequalities hold for the case of a projective variety and
an ample line bundle (with identical proofs).
\end{Rem}

Using the classical Brunn-Minkowski inequality
it is elementary to see that the function
$\Delta \mapsto \log(\Vol(\Delta))$ is a concave function on the
space of convex bodies in $\Bbb R^n$. Let $Y$ be a projective variety with a
line bundle $\L$. From analogue of Theorem \ref{17.1} for $(Y, \L)$ it follows that
\begin{Cor}[Log-concavity of degree of line bundles]
The function $$\L \mapsto \log(\deg(\L)),$$ is a concave function.
\end{Cor}

\subsection{Algebraic proof of the geometric inequalities} \label{subsec-18}
\begin{Th} \label{18.1}
All geometric inequalities from
Section \ref{subsec-16} follow from their algebraic analogues in Section \ref{subsec-18}.
\end{Th}
\begin{proof} It it enough to prove the Alexandrov--Fenchel
inequality. For polyhedra with integral vertices it follows from its
algebraic analogue and from Bernstein theorem which states that the
mixed volume of such polyhedra is equal to the intersection number
of spaces $L_i$ in $K(X)$ where $X=(\Bbb C^*)^n$ and $L_i$ are
spaces of Laurent polynomials with fixed Newton polyhedra (Section \ref{subsec-15}).
The inequality for the convex polyhedra with rational vertices
follows from this result because, after multiplication by an
appropriate number, one turn a polyhedron with rational
vertexes into a polyhedron with integral vertices.
Now it is enough to use the inequality for polyhedra with integral vertexes and
multi-linearity of the mixed volume.
Finally any convex body can be approximated, in the Hausdorff
metric, by polyhedra with rational vertices.  Since mixed volume is
continuous with respect to the topology induced by the Hausdorff
metric the theorem is proved.
\end{proof}

Note that an algebraic proof of the Alexandrov--Fenchel inequality
has been known earlier
(see \cite[Addendum 3, Algebra and mixed volumes, pp. 182-207]{Burago-Zalgaller}).
But a crucial step in that proof is the use of
Hodge Index Theorem. In fact the proof of Hodge Index Theorem itself is as
complicated as the Alexandrov--Fenchel inequality.
Here we have given rather simple proofs of both of this two theorems (Hodge
Index Theorem and Alexandrov--Fenchel inequality) simultaneously using just two
very classical results: Brunn--Minkowski inequality (and in fact it is even enough to use the isoperimetric
inequality in the plane) in the geometric side and the Hilbert theorem on degree in
the algebraic side.

\subsection{Newton convex body in the case of curves} \label{subsec-19}
In this section we assume that $X$ is an affine irreducible curve.
As before we say that the space $L\in
K(X)$ is {\it ample} if the mapping degree for the natural
map $\Phi_L: X\rightarrow \Bbb P(L^*)$ is $1$, in other words if $\Phi_L$,
restricted to an open dense subset, is an embedding.
Using Hilbert theorem and Riemann--Roch theorem one can find
a good  estimate for the dimension of the space $L^k$ for large values of $k$.

\begin{Prop} \label{19.1}
Assume that space $L\in K(X)$ is
ample. Then there exists $k_0 \geq 0$
and $C \geq 0$ such that for $k>k_0$ the dimension of the space $L^k$ is
equal to $k \deg L + C$. Moreover one can take $C \leq 1-g$ where $g$
is the genus of the curve $X$.
\end{Prop}
\begin{proof} The dimension of the curve  $X$ is $1$ and the number of zero of a
generic function $f\in L^k$
is equal to $\deg L^k=k \deg L$ (see the proposition 5.1). So by
Hilbert theorem, for sufficiently big  $k$, the dimension of the space
$L^k$ is equal to $k \deg L +C$. Consider a compactification $\bar
X$ of $X$. At any point $a\in A =\bar X
\setminus X$ we have $ord_a L^k =k ord_a L$. By
Riemann--Roch theorem if $ord _a L^k>2g-2$  the dimension of the
space of all regular functions on $X$ whose order at the point $a \in A$
is $\geq ord _a L^k$ is equal to $\sum _{a\in A} ord
_aL^k -g =k\dim L-g+1$.
\end{proof}
\begin{Def} Let $a \in X$. The {\it
valuation corresponding to $a$} is the valuation $v_a$ on $\Bbb C(X)$
defined by $v_a(f) = ord_af$ for any $f \in \Bbb C(X)$.
This valuation take values in $\Bbb Z$ is faithful: for each integer $m\in \Bbb Z$ there is a
rational function $f$ with $ord_af=m$.
\end{Def}

For an ample space $L$ and for the valuation $v_a$ corresponding to
a point $a\in X$ one can find a good description of the graded
semigroup  $G(L)$ and the Newton convex body $\Delta (G(L))$. For a
$k > 0$ denote by $G_k(L)$ the set of degree $k$
elements in the semigroup $G(L)$. Denote by $\tilde G(L)$ the
projection of the $G(L)$ on the valuation line.

\begin{Th} \label{19.2}
If a space $L$ is ample, then for
a valuation related to a point $a\in X$ the following holds.
\begin{enumerate}
\item The semigroup $\tilde G(L)$ generates the group $\Bbb Z$.
\item The projection of the Newton convex body $\Delta(G(L))$ of the graded
semigroup $G(L)$ on the valuation line is the segment $[0, \deg L]$.
For any $k>0$, each point $(k,m)$ in the set $G_k(L)$  satisfies  the
inequalities $0\leq m\leq k\deg L$.
\item There is a constant $C_0$ and a function $C_1(k)$ such that $\lim _{k\rightarrow \infty}\frac{
C_1(k)}{k} =0$ and the set $G_k(L)$ contains all points
$(k,m)$ satisfying $ C_0\leq m\leq \deg
L -C_1$.
\end{enumerate}
\end{Th}
\begin{proof} Since $L$ is ample the map $\Phi_L: X\rightarrow \Bbb P(L^*)$ induces
a birational isomorphism between the curve $X$ and its image $Y$.
This means that each rational function on $X$ is pull-back of
a rational function on $Y$. So $\tilde G(L) = \Bbb Z$.
By Theorem \ref{14.4} (main theorem) the number of roots on $X$ of a
generic function $f\in L$ is equal to the length of the  Newton
segment $\Delta (G(L))$ (note that a $1$-dimensional bounded convex
domain is a segment). On the other hand this number of roots is
equal to $\deg L$ (see Proposition \ref{5.1}). So the length  of the
segment $\Delta (G(L))$ is equal to $\deg L$.

The space $L$ contains a function $f$ with $f(a)\neq 0$. So the
semigroup $G(L)$ contains the point $e=(1,0)$. Hence the projection of
the Newton segment on the valuation line contains the origin and
it coincides with the segment $[0,\deg L]$. As
we proved above (see ...) there are two non-negative functions $C_0(k)$
$C_1(k)$ such that $\lim _{k\rightarrow \infty}\frac{
C_0(k)}{k}= \lim _{k\rightarrow \infty}\frac{ C_1(k)}{k} =0$ and
any point satisfying the inequalities $ C_0(k)\leq m\leq \deg L
-C_1(k)$ belongs to $G_k(L)$. For $k_2>k_1$, by adding vector $(k_2 - k_1)$,
we can embed $G_{k_1}(L)$ into $G_{k_2}(L)$. So the condition $ C_0(k)\leq m$
can be replaced by $C_0\leq m$, where $C_0$ is a sufficiently large
positive constant.
\end{proof}

Fix a valuation corresponding to some point $a\in X$.
\begin{Cor} \label{19.3}
If the space $L$ is ample, then
the number $[L]$ of roots of a sufficiently general function $f\in
L$ is equal to the length of the Newton segment $\Delta (G(L))$. The
Newton segment $\Delta (G(L_1L_2))$ of the product of two ample
enough spaces $L_1$ and $L_2$ is equal, up to a shift, to the sum $\Delta
(G(L_1))+\Delta (G(L_2))$ of the Newton segments of $L_1$
and $L_2$.
\end{Cor}

The semigroup $G(L)$ belongs to a cone over a segment which
contains points $(1, x)$ where $0\leq x \leq \deg L$.
Clearly the cone of the semi-group $G(L)$ consists of two rays one of which
is the upper-half of the vertical axes. The following question is important for us:
does the semigroup  $G(L)$ contain an integral point on the other boundary ray of
its cone? that is, does $G(L)$ contain an integral point at the ray $\lambda(1,
\deg L)$, $\lambda>0$? If
the genus $g$ of the curve $X$ is positive, then as a rule the
answer to the question is negative. Indeed we have the following.
let $D$ be a divisor of poles of the space $L$ and $a$  be the
point corresponding to the valuation.

\begin{Prop} \label{19.4}
If the semigroup $G(L)$ contains a
point $(k, k \deg L)$ where $k>0$ then the divisors $k D$ and
$(k\deg L) a$ are linearly equivalent.
\end{Prop}
\begin{proof} If the point $(k, k \deg L)$ belongs to
$G(L)$ then there is $f\in L^k$ such that the  divisor of
its poles is $kD$ and the divisor of its zeros is $(k\deg L) a$. The
existence of such a function means that $k(\deg L) a-kD$ is a
principal divisor.
\end{proof}

\begin{Cor} \label{19.5}
If the genus $g$ of $X$ is positive and
if the space $L\in K(X)$ is ample then for almost all points
$a\in X$, the semigroup $G(L)$, for the valuation corresponding to $a$,
is not finitely generated.
\end{Cor}
\begin{proof} If a semigroup is generated by a finite set
$M=\{(k,m)\}$, where $k>0$, then a point in  $M$ at which the quotient
$m/k$ attains its maximum belongs to the boundary of the minimal
convex cone which contained the semigroup. By the proposition 19.4
it could happen only if for some natural $k$ the divisor $(k\deg
L)a$ is equivalent to the divisor $k D$.   At most countable set of
points $a\in X$  could satisfy this condition if the genus $g$ is
positive.
\end{proof}

Let $L$ be an element in the semigroup $K(X)$ and let $Y=\Phi_L(X)$ be the image of
$X$ in $\Bbb P(L^*)$.
Let $d$ denote the mapping degree of $\Phi_L: X \to Y$
and let $\mu_a$ be the local mapping degree of $\Phi_L$ at a point $a$.
Fix the valuation related to the point $a$.

\begin{Th} \label{19.6}
With notation as above, the following
are true:
\begin{enumerate}
\item The  group generated by the semigroup $\tilde G(L)$ is a
subgroup of index $\mu_a$ in $\Bbb Z$.
\item The projection of the Newton convex body $\Delta(G(L))$ of the graded
semigroup $G(L)$ on the valuation line is the segment $[0,
\frac{\mu_a \deg(L)}{d}]$. For any $k>0$, every point $(k,m)$ in
the set $G_k(L)$  satisfies  the inequality $$0\leq m\leq
\frac{k\mu_a \deg L}{ \deg},$$ moreover the number $m$
is divisible by $\mu_a$.
\item There is a constant $C_0$ and there is a function $C_1(k)$ of
such that $$\lim _{k\rightarrow \infty}\frac{
C_1(k)}{k} =0,$$ and the set $G_k(L)$ contains all points
$(k,m)$ such that $m$ is divisible by $\mu_a$ and $ C_0\leq m\leq \deg L -C_1$.
\end{enumerate}
\end{Th}
\begin{proof} The field of rational functions on $Y$
is isomorphic to the field of rational functions on a non singular
model $\tilde{Y}$ of the curve $Y$. The map  $\Phi_L$ can be lifted to
the map $\tilde \Phi_L :X \rightarrow \tilde{Y}$.
The germ of the curve $X$ at the point $a$ covers the germ of
the curve  $\tilde{Y}$ at the point $\tilde \Phi_L(a)$ with the multiplicity
$\mu_a$. Because of this the index of the subgroup of $\Bbb Z$
generated by the semigroup $\tilde G(L)$  is equal to $\mu_a$.

By Theorem \ref{14.4} (main theorem) the number of zeros on $X$ of a
generic function  $f\in L$ is equal to the length of the Newton
segment $\Delta(G(L))$ multiplied by the number $d/\mu_a$. On
the other hand this number of roots is equal to $\deg(L)$. So the
length of the segment $\Delta (G(L))$ is equal to $\frac{\mu \deg
(L)}{d}$. To finish the proof use the same arguments in the proof of
Theorem \ref{19.2}.
\end{proof}

\begin{Cor} \label{19.7}
For each space $L\in K(X)$ and for
the valuation corresponding to a point $a\in X$ we have $d \geq \mu$,
where $d$ is the mapping degree of $\Phi_L:X \to \Bbb P(L^*)$
and $\mu$ is the index of the subgroup of $\Bbb Z$
generated by generated by the semigroup $\tilde G(L)$.
\end{Cor}

\begin{Rem} 1) In general the numbers $d$ and $ind$ in the
main theorem are different. But when $X$ is a so-called spherical variety
for the action of a connected reductive algebraic group $G$ and the space $L\in K(X)$
is invariant under $G$ these two numbers coincide (see \cite{Khovanskii-Kaveh}).
2) The inequality in Corollary \ref{19.7}
can be easily extended to the general higher dimensional case.
\end{Rem}

\subsection{Degeneration of a variety to a toric variety and SAGBI bases} \label{subsec-20}
Let $A = \Bbb C[x_1,\ldots,x_n]$.
Fix a term ordering $<$ on $\Bbb Z^n$ which we regard as the semi-group of monomials
in $A$. As usual define the initial term map $v: R \to \Bbb Z^n$ as follows:
let $f \in A$ and
let $cx_1^{\alpha_1} \cdots x_n ^{\alpha_n}$ be the lowest term of $f$ with respect
to $<$. Put $v(f) = (\alpha_1, \ldots, \alpha_n)$. $v$ extends to a valuation
on the field of rational polynomials $\Bbb C(x_1, \ldots, x_n)$.
Let $R$ be a subring of $A$. $R$ is said to have a SAGBI basis
(subalgebra analogue of Gr\"{o}bner basis for ideals), with respect to $<$,
if the semi-group
of initial terms $v(R)$ is finitely generated. A set of polynomials $\{f_1, \ldots,
f_r\} \subset R$ such that $v(f_1), \ldots, v(f_r)$ is a set of generators
for the semi-group $v(R)$ is called a SAGBI basis. The remarkable fact about
a SAGBI basis is that any element of $R$ can be represented as a polynomial
in the $f_i$ by a classical simple finite algorithm called {\it subdection algorithm}.

Below we generalize the notion of SAGBI basis to a finitely generated subalgebra
of the coordinate ring of a (quasi) affine variety.
Let $X$ be a (quasi) affine variety of dimension $d$.
Fix a term order on $\Bbb Z^n$ and let $v: \Bbb C(X) \to \Bbb Z^n$ be a valuation with
respect to $<$. As usual let $L$ be a finite dimensional subspace of $\mathcal{O}(X)$ and
put $R = \bigoplus_{k=0}^\infty L^k$.
\begin{Def}
$R$ is said to have a SAGBI basis, with respect to $v$, if $v(R)$ is a finitely generated
semi-group. Similarly, let $Y$ be a projective variety and $\L$ a line bundle.
Put $\mathcal{R} = \bigoplus_{k=0}^\infty H^0(Y, \L^k)$. $\mathcal{R}$ is said to have a SAGBI basis
if $v(\mathcal{R})$ is a finitely generated semi-group.
\end{Def}
Obviously if $R$ (respectively $\mathcal{R}$)
has a SAGBI basis then the valuation cone of $v$ is a
convex polyhedral cone and the convex body $\Delta(G(L))$
(respectively $\Delta(G(\L))$) is
a polytope. Conversely, if the valuation cone of $v$ is polyhedral and
moreover if $v(R)$ (respectively $v(\mathcal{R})$)
coincides with all the integral points in the valuation cone then
it is a finitely generated semi-group and hence $R$ (respectively $\mathcal{R}$)
has a SAGBI basis. This is the case in many important examples namely,
toric varieties, flag varieties and spherical varieties of a
complex connected reductive group $G$(See \cite{Kaveh}).

When the homogeneous
coordinate ring $\mathcal{R}$ of a projective variety $Y \subset \Bbb P(V)$
has a SAGBI basis,
it follows from \cite[Theorem 15.17]{Eisenbud} that
$Y$ can be degenerated to a toric variety, that is,
there is a flat family $Y_t, t \in \Bbb C$ of subvarieties of $\Bbb P(V)$ such that
$Y_t$ is isomorphic to $Y$ for $t \neq 0$ and $Y_0$ is a toric variety
(with the same dimension as $Y$ of course).

\noindent Askold G. Khovanskii\\Department of
Mathematics\\University of Toronto \\Toronto, ON M5S 2E4\\Canada\\
{\it Email:} {\sf askold@math.utoronto.ca}\\

\noindent Kiumars Kaveh\\Department of Mathematics\\University of
Toronto\\Toronto, ON M5S 2E4\\Canada\\
{\it Email:} {\sf kaveh@math.utoronto.ca}\\

\begin{thebibliography}{99}
\bibitem[Bern]{Bernstein} Bernstein, D. N.
{\it The number of roots of a system of equations}.
English translation: Functional Anal. Appl. 9 (1975), no. 3, 183--185 (1976).

\bibitem[Burago-Zalgaller]{Burago-Zalgaller} Burago, Yu. D.; Zalgaller, V. A.
{\it Geometric inequalities}. Translated from the Russian by A. B. Sosinski\u\i.
Grundlehren der Mathematischen Wissenschaften, 285.
Springer Series in Soviet Mathematics (1988).

\bibitem[B-Kh]{Burda-Khovanskii} Burda, Y.; Khovanskii, A. G. {Degree of rational mapping and theorems of
Sturm and Tarski}. Preprint (2008).

\bibitem[Eisenbud]{Eisenbud} Eisenbud, D.
{\it Commutative algebra. With a view toward algebraic geometry}. Graduate Texts in Mathematics, 150. Springer-Verlag, New York, 1995.

\bibitem[Gelfand-Cetlin]{G-C} Gelfand, I. M.; Cetlin, M. L. {\it Finite-dimensional
representations of the group of unimodular matrices}. (Russian)
Doklady Akad. Nauk SSSR (N.S.)  71,  (1950). 825--828.

\bibitem[Harris]{Harris} Harris, J.
{\it Algebraic geometry. A first course}. Graduate Texts in Mathematics, 133. Springer-Verlag, New York, 1992.

\bibitem[Hart]{Hartshorne} Hartshorne, R. {\it Algebraic geometry}.
Graduate Texts in Mathematics, No. 52. Springer-Verlag, New York-Heidelberg, 1977.

\bibitem[Hodge-Pedoe]{Hodge-Pedoe} Hodge, W. V. D.; Pedoe, D.
{\it Methods of algebraic geometry}.
Vol. III. Book V: Birational geometry. Cambridge, at the University Press, 1954.

\bibitem[Kaveh]{Kaveh} Kaveh, K. {\it Newton polytopes for flag and
spherical varieties}. Preprint (2008).

\bibitem[K-Kh]{Khovanskii-Kaveh} Kaveh, K.; Khovanskii, A.G. {\it Bernstein theorem for varieties with a reductive group action.} Preprint (2008).

\bibitem[Khov1]{Khovanskii-finite-sets} Khovanskii, A. G. {\it Sums of finite sets, orbits of
commutative semigroups and Hilbert functions}. (Russian)
Funktsional. Anal. i Prilozhen.  29  (1995),  no. 2, 36--50, 95;
translation in  Funct. Anal. Appl.  29  (1995),  no. 2, 102--112.

\bibitem[Khov2]{Khovanskii-formulas} Khovanskii, A.G. {\it Geometry of Formulas}, (pp. 67-91, Section 3)
in  V.I. Arnold , A.N. Varchenko , A.B. Givental and A.G.Khovanskii
"Singularities of functions, wave fronts, caustics and multidimensional
integrals", pp. 1-91 in Soviet Scientific Reviews, Section C, MATHEMATICAL
PHYSICS REVIEWS, Volume 4 (1984).

\bibitem[Kush]{Kushnirenko} Ku\v{s}nirenko, A. G.
{\it Polyèdres de Newton et nombres de Milnor}. (French)
Invent. Math. 32 (1976), no. 1, 1--31.

\bibitem[Litt]{Littelmann} Littelmann, P. {\it Cones, crystals, and patterns}.
Transform. Groups 3 (1998), no. 2, 145--179.

\bibitem[Ok1]{Okounkov-Brunn-Minkowski} Okounkov, A.
{\it Brunn-Minkowski inequality for multiplicities}.
Invent. Math. 125 (1996), no. 3, 405--411.

\bibitem[Ok2]{Okounkov-Newton-polytope} Okounkov, A. {\it Multiplicities and Newton polytopes}
Kirillov's seminar on representation theory, 231–244, AMS Transl.
Ser. 2, 181, AMS, Providence, RI, 1998. 21.

\bibitem[Par]{Parshin} Parshin, A. N.
{\it Local class field theory}. (Russian)
Algebraic geometry and its applications.
Trudy Mat. Inst. Steklov. 165 (1984), 143--170.

\bibitem[T]{Teissier} Teissier, B.
{\it Du théorème de l'index de Hodge aux inégalités isopérimétriques}.
C. R. Acad. Sci. Paris Sér. A-B 288 (1979), no. 4, A287--A289.

\bibitem[Dries]{Dries} van den Dries, L.
{\it Tame topology and o-minimal structures}.
London Mathematical Society Lecture Note Series, 248. Cambridge University Press, Cambridge, 1998.

\bibitem[Wh]{Whitney} Whitney, H.
{\it Elementary structure of real algebraic varieties}.
Ann. of Math. (2) 66 1957 545--556.
14.00


\end{thebibliography}
\end{document}